\documentclass[a4paper]{article}
\usepackage[all]{xy}\usepackage[latin1]{inputenc}        %accents
\usepackage[dvips]{graphics,graphicx}
\usepackage{amsfonts,amssymb,amsmath,color,mathrsfs, amstext}
\usepackage{amsbsy, amsopn, amscd, amsxtra, amsthm,authblk,enumerate,algorithmicx,algorithm}
\usepackage{algpseudocode}
\usepackage{upref}
\usepackage{geometry}
\usepackage[displaymath]{lineno}
\usepackage{float}
\usepackage{yhmath}
\usepackage[colorlinks,
            linkcolor=red,
            anchorcolor=red,
            citecolor=red
            ]{hyperref}
\usepackage{subfigure}
\numberwithin{equation}{section}

\def\epsilon{\varepsilon}

\newcommand{\E}{\mathbb{E}}
\newcommand{\R}{\mathbb{R}}
\newcommand{\bbP}{\mathbb{P}}

\newcommand{\cL}{\mathcal{L}}

\newtheorem{theorem}{Theorem}[section]
\newtheorem{lemma}{Lemma}[section]
\newtheorem{proposition}{Proposition}[section]
\newtheorem*{proposition*}{Proposition}

\newtheorem*{corollary*}{Corollary}

\newtheorem*{definitions*}{Definitions}
\newtheorem*{conjecture*}{\bf Conjecture}

\theoremstyle{remark}
\newtheorem{remark}{\bf Remark}[section]

\newtheorem{assumption}{Assumption}[section]

\newcounter{parentalgorithm}

\makeatother

\numberwithin{equation}{section}

\begin{document}

\title{On the Random Batch Method for second order interacting particle systems}
\author[1]{Shi Jin \thanks{shijin-m@sjtu.edu.cn}}
\author[2]{Lei Li\thanks{leili2010@sjtu.edu.cn}}
\author[3]{Yiqun Sun\thanks{yqsun5@sjtu.edu.cn}}
\affil[1,2,3]{School of Mathematical Sciences, Shanghai Jiao Tong University, Shanghai, 200240, P. R. China.}
\affil[1,2]{Institute of Natural Sciences, MOE-LSC, Shanghai Jiao Tong University, Shanghai, 200240, P. R. China.}

\date{}
\maketitle

\begin{abstract}
We investigate several important issues regarding the Random Batch Method (RBM) for second order interacting particle systems. We first show the uniform-in-time strong convergence for second order systems under suitable contraction conditions. Secondly, we propose the application of RBM for singular interaction kernels via kernel splitting strategy, and investigate numerically the application to molecular dynamics.
\end{abstract}

\section{Introduction}

A significant number of important phenomena in physical, social, and biological sciences are described at the microscopic level by interacting particle systems, which exhibit interesting features. Examples include fluids and plasma \cite{frenkel2001understanding,birdsall2004},
swarming \cite{vicsek1995novel,carrillo2017particle,carlen2013kinetic,degond2017coagulation}, chemotaxis \cite{horstmann03,bertozzi12}, flocking  \cite{cucker2007emergent,hasimple2009,albi2013}, synchronization \cite{choi2011complete,ha2014complete}
 and consensus \cite{motsch2014}, to name a few.
These interacting particle systems can be described in general by the first order systems
\begin{gather}\label{eq:Nparticlesys}
dX^i=b(X^i)\,dt+\alpha_N\sum_{j: j\neq i} K(X^i-X^j)\,dt+\sigma\, dW^i,~~i=1,2,\cdots, N,
\end{gather}
or the second order systems
\begin{gather}\label{eq:Nbody2nd}
\begin{split}
& dX^i=V^i\,dt,\\
& dV^i=\Big[ b(X^i)+\alpha_N \sum_{j:j\neq i}K(X^i-X^j)-\gamma V^i \Big]\,dt+\sigma\, dW^i.
\end{split}
\end{gather}
Here, $X^i\in \R^d$ are the labels for the particles, and $b(\cdot)$ is some given external field. The stochastic processes $\{W^i\}_{i=1}^N$
are i.i.d. Wiener processes, or the standard Brownian motions.
We will loosely call $X^i$ the ``locations'' or ``positions'',
and $V^i$ the velocities of the particles, though the specific meaning can be different in different situations. The function $K(\cdot): \R^d\to \R^d$ is the interaction kernel. If $\gamma=\sigma=0$ and $b=-\nabla U$ for some potential $U$, one has a Hamiltonian system like the one for electrons in plasma \cite{vlasov1961many}.
For the molecules in the heat bath \cite{kawasaki1973simple,callen1951irreversibility}, $X^i$ and $V^i$ are the physical positions and velocities, described by the underdamped Langevin equations, where $\sigma$ and $\gamma$ satisfy the so-called ``fluctuation-dissipation relation"
\begin{gather}
\sigma=\sqrt{2\gamma/\beta},
\end{gather}
where $\beta$ is the inverse of the temperature (we assume all the quantities are scaled and hence dimensionless so that the Boltzmann constant is absent). The first order system \eqref{eq:Nparticlesys} can be viewed as the overdamped limit of the second order systems \eqref{eq:Nbody2nd}.

If one directly discretizes \eqref{eq:Nparticlesys} or \eqref{eq:Nbody2nd}, the computational cost per time step is $\mathcal{O}(N^2)$. This is undesired for large $N$. The Fast Multipole Method (FMM) \cite{rokhlin1985rapid} is able to reduce the complexity to $\mathcal{O}(N)$ for fast enough decaying interactions. However, the implementation of FMM is quite involved. A simple random algorithm, called the Random Batch Method (RBM), has been proposed in \cite{jin2020random}  to reduce the computation cost per time step from $\mathcal{O}(N^2)$ to $\mathcal{O}(N)$, based on the simple ``mini-batch" idea. The ``random mini-batch'' idea is famous for its application in the so-called stochastic gradient descent (SGD) \cite{robbins1951,bottou1998online,bubeck2015convex} for machine learning problems. The idea was also used for Markov Chain Monte Carlo methods like the stochastic gradient Langevin dynamics (SGLD) by Welling and Teh \cite{welling2011bayesian} and the Random Batch Monte Carlo methods \cite{lixuzhao2020}, and also for the computation of the mean-field flocking model \cite{albi2013,carrillo2017particle} motivated by Nanbu's algorithm of the Direct Simulation Monte Carlo method \cite{bird1963approach,nanbu1980direct,babovsky1989convergence}. 
The key behind the ``mini-batch'' idea is to find some cheap unbiased random estimator for the original quantity with the variance being controlled. Depending on the specific applications, the design can be different. For interacting particle systems in \cite{jin2020random}, this is realized by random grouping and then interacting the particles only within the groups for each small time subinterval. Compared with FMM, the accuracy of RBM is lower, but RBM is much simpler and is valid for more general potentials (e.g. the SVGD ODE \cite{li2019stochastic}). 
The method converges due to the time average in time, and thus the convergence is like that in the Law of Large Number, but in time (see \cite{jin2020random}  for more detailed explanation). Hence, one may understand such methods as certain Monte Carlo methods. If there is mixing and ergodicity for the systems, the simulation can converge well.

RBM for interacting particle systems has been used or extended in various directions,
from sampling \cite{li2019stochastic,lixuzhao2020,jinli2020qmc} to molecular dynamics \cite{jinlixuzhao2020rbe,li2020direct}, and control of synchronization \cite{biccari2020stochastic,ko2020model}.
RBM has been shown to converge for finite time interval if the interaction kernels are good enough \cite{li2019stochastic,jin2020random}, and in particular an error analysis for deterministic Newton type second order systems is obtained in the Appendix of \cite{jin2020random}.  Moreover, a convergence result of RBM for $N$-body Schr\"odinger equation is also obtained in \cite{golse2019random}. For long time behaviors, it is expected that the method works for systems that own ergodicity and mixing properties, like systems in contact with heat bath and converge to equilibria. Previous rigorous studies of such type mainly focus on first order systems due to good contraction and mixing properties \cite{jin2020random,jin2020convergence}.  
Second order systems are however more common in nature, especially systems in contact with heat bath that are very important for molecular dynamics \cite{frenkel2001understanding}. Whether RBM can be applied directly to obtain good results for direct molecular dynamics simulation needs careful study both in theory and in practice. 
For closed systems that are Hamiltonian, like the particle systems for the Vlasov-Poisson equations (in this case $\gamma=\sigma=0$), RBM may be applied to get correct simulation for finite time, but the long time behavior is not clear for these systems. Hence, in this work we mainly focus on systems that are in contact with heat bath. 

In this work, our goal is two-folded. Firstly, we aim to prove rigorously that RBM converges for large times with certain contraction conditions for second order systems \eqref{eq:Nbody2nd}. Secondly, we aim to combine the random grouping strategy with the kernel strategy as in \cite{martin1998novel,hetenyi2002multiple,lixuzhao2020} so that RBM could be practically applied for molecular dynamical simulations.

Now let us remark the regimes to consider. In the mean field limit regime (\cite{stanley1971, georges1996, lasry2007}), one chooses 
\begin{gather}
\alpha_N=\frac{1}{N-1}
\end{gather}
 so that as $N\to\infty$ the empirical distribution $\mu^{(N)}:=N^{-1}\sum_{i=1}^N \delta(x-X^i)\otimes\delta(v-V^i)$ converges almost surely under the weak topology to the solutions of the limiting PDE 
\begin{gather}\label{eq:limitvlasov}
\partial_t f=-\nabla_x\cdot(vf)-\nabla_v\cdot\Big((b(x)+K*_xf-\gamma v)f\Big)+\frac{1}{2}
\sigma^2\Delta_v f.
\end{gather}
The particle system \eqref{eq:Nbody2nd} can also be regarded as a numerical particle method for solving these mean field PDE \eqref{eq:limitvlasov}. Examples of such PDEs include the granular media equations \cite{cattiaux2008}
and the Vlasov equations for which $\gamma=\sigma=0$
\cite{vlasov1961many}.  In \cite{jin2020random}, it has been shown that
RBM is asymptotic-preserving for first order systems regarding the mean-field limit, which means the algorithm can approximate the one-marginal distribution with error bound independent of $N$. Below, we show in section \ref{sec:esterr} that RBM is also asymptotic-preserving regarding the mean-field limit for second order systems under suitable conditions.

In the molecular dynamics simulations,  one chooses $\alpha_N=1$, and the equations are basically given by
\begin{gather}\label{eq:mdparticle}
\begin{split}
& dX^i=V^i\,dt ,\\
& dV^i=\Big[-\sum_{j:j\neq i}\nabla \phi(X^i-X^j)\Big]\,dt+d\xi^i.
\end{split}
\end{gather}
Here, $\phi(\cdot)$ is the interaction potential and $d\xi^i$ means the interaction with the environment that changes the momentum, which we will discuss in section \ref{subsec:heatbath}.
Depending on how to model the coupling to heat batch, one may choose different thermostats like the Andersen thermostat, the Langevin dynamics or the Nos\'e-Hoover thermostat etc so different expressions for $d\xi^i$ can be used (see section \ref{subsec:heatbath}). Though $\alpha_N=1$ is often chosen for molecular dynamics, one may do time and spatial rescalings to match the mean field regime $\alpha_N=1/(N-1)$ factor. However, the scaling is not crucial for simulation of molecular dynamics (see the discussion in section \ref{subsec:dis}), hence, in this molecular dynamics regime with $\alpha_N=1$, we will apply RBM directly when it has benefits without scaling it to the mean-field regime.

The rest of the paper is organized as follows. In section \ref{sec:alg}, we give a brief introduction to RBM and introduce the potential splitting so that RBM can be applied for systems with singular interaction kernels.
In section \ref{sec:esterr}, we establish the long time strong error estimate
for regular kernels under certain contraction conditions. We provide some discussions on the details on applying RBM with kernel splitting to simulations of molecular dynamics in section \ref{sec:md}. 
Some numerical experiments are performed in section \ref{sec:numerics} to verify the claims and validate the methods.

\section{The algorithms}\label{sec:alg}
In this section, we give a detailed explanation of RBM for interacting particle systems and then propose the application of RBM for singular interaction kernels via a kernel splitting strategy.

Let us briefly explain the random grouping strategy for RBM in \cite{jin2020random} that realizes the mini-batch idea for interacting particle systems. 
Let $T>0$ be the simulation time, and choose a time step $\tau>0$. Pick a batch size $p\ll N$, $p\ge 2$ that divides $N$ (RBM can also be applied if $p$ does not divide $N$; we assume this only for convenience). Consider the discrete time grids $t_k:=k\tau$, $k\in \mathbb{N}$. For each subinterval $[t_{k-1}, t_k)$, the method has two substeps: (1) at $t_{k-1}$, divide the $N$ particles into $n:=N/p$ groups (batches) randomly; (2) let the particles evolve with interaction only inside the batches.

\subsection{RBM for regular kernels}

Applying the above strategy to the second order system \eqref{eq:Nbody2nd} with interacting forces that do not have singularity  yields the method as shown in Algorithm \ref{alg:rbm2nd1}.
\begin{algorithm}[H]
\caption{(RBM for \eqref{eq:Nbody2nd})}
\label{alg:rbm2nd1}
\begin{algorithmic}[1]
\For {$m \text{ in } 1: [T/\tau]$}   
\State Divide $\{1, 2, \ldots, N=pn\}$ into $n$ batches randomly.
     \For {each batch  $\mathcal{C}_q$} 
     \State Update $X^i$'s ($i\in \mathcal{C}_q$) by solving  for $t\in [t_{m-1}, t_m)$ the following
     \begin{gather}\label{eq:RBM2nd}
     \begin{split}
            & dX^i=V^i\,dt,\\
            & dV^i=\Big[b(X^i)+\frac{\alpha_N(N-1)}{p-1}\sum_{j\in\mathcal{C}_q,j\neq i}K(X^i-X^j) -\gamma V^i\Big]\,dt+\sigma\, dW^i.
      \end{split}
      \end{gather}
      \EndFor
 \EndFor
\end{algorithmic}
\end{algorithm}

The method shown in Algorithm \ref{alg:rbm2nd1} shares some similarity with the Stochastic Gradient Hamiltonian Monte Carlo (SGHMC) with friction proposed in \cite[sections 3.2-3.3]{chen2014stochastic}, which is a Markov Chain Monte Carlo method for Bayesian inference and machine learning. The difference is that the method shown in Algorithm \ref{alg:rbm2nd1}
uses random grouping for interacting particles, while SGHMC uses random samples to compute the approximating gradients; i.e., the ways to implement mini-batch are different. The SGHMC in \cite{chen2014stochastic} is a sampling method and the momentum will be resampled occasionally. Since the underdamped Langevin system does not satisfy the detailed balance so using it as a block for the Markov chain may yield in some systematic error.
The method shown in Algorithm \ref{alg:rbm2nd1} is a direct simulation approach for the underdamped Langevin equation so it can be used both for dynamical simulation to capture the transition behaviors approximately, and can also be used for sampling from the equilibrium.

\begin{remark}
Despite the difference mentioned above, we remark that the random grouping strategy can be viewed as a particular stochastic gradient as in \cite{robbins1951,welling2011bayesian} when $K(x)=-\nabla\phi(x)$.
In fact, for this case, we introduce the $Nd$-dimensional vector $\mathfrak{X}:=(X^1,\ldots, X^N)$ and consider the full interacting energy corresponding to $\alpha_N=1/(N-1)$
\begin{gather}\label{eq:energy}
E(\mathfrak{X}):=\frac{1}{2(N-1)}\sum_{i\neq j}\phi(X^i-X^j).
\end{gather}
In \cite{robbins1951,welling2011bayesian}, the stochastic gradient can be computed by choosing any subset of terms in the sum \eqref{eq:energy}. For random grouping in  \cite{jin2020random}, one is only allowed to choose the summands in a particular way. For a given set of random batches $\mathcal{C}=\{\mathcal{C}_1,\cdots, \mathcal{C}_n\}$, one may use the following random variable 
\begin{gather}
\tilde{E}(\mathfrak{X}):=\frac{1}{2(p-1)}\sum_{q=1}^n\sum_{k, \ell \in \mathcal{C}_q}\phi(X^k-X^{\ell})
\end{gather}
to approximate $E(\mathfrak{X})$ and using its gradient for the dynamics leads to the random grouping in \cite{jin2020random}
in the case $K=-\nabla\phi$, though RBM in \cite{jin2020random} applies to more general kernels.
\end{remark}

\subsection{RBM with kernel splitting}

If the interaction kernel $K$ is singular at $x=0$ which is often the case in applications, direct discretization of the equations in Algorithm \ref{alg:rbm2nd1} can lead to numerical instability. For first order systems, in the case $p=2$, one may take advantage of the time-splitting method to accurately solve the singular part to eliminate the instability \cite{jin2020random,li2020direct}. For second order systems or first order systems with $p\ge 3$, the time splitting trick does not apply any more, and applying RBM directly leads to poor results. To resolve this issue, we adopt the splitting strategy in \cite{martin1998novel,hetenyi2002multiple,lixuzhao2020}.

In fact, one decomposes the interacting force $K$ into two parts:
\begin{gather}
K(x)=K_1(x)+K_2(x).
\end{gather}
Here, $K_1$ has short range that vanishes for $|x|\ge r_0$ where $r_0$ is a certain cutoff chosen to be comparable with the mean distance of the particles. The part $K_2(x)$ is a bounded smooth function.
With this decomposition, we then apply RBM to the $K_2$ part only. The resulted method is shown in Algorithm \ref{alg:rbms}.
Now, the summing in $K_1$ can be done in $\mathcal{O}(1)$ time for given $i$ due to the short range. This can be implemented using data structures like Cell-List \cite[Appendix F]{frenkel2001understanding}.
Hence, the cost per time step is again $\mathcal{O}(N)$. Since $K_2$ is bounded, RBM can be applied well due to the boundedness of variance, without introducing too much error. For practical applications, $K_1$ is a repulsive force so that computing the $K_1$ part accurately will forbid the particles getting too close so that the system is not stiff. Then, numerical simulations can be performed well. We show some numerical results in section \ref{subsec:ljnum}.

\begin{algorithm}[H]
\caption{RBM with splitting for \eqref{eq:Nparticlesys} and \eqref{eq:Nbody2nd} }
\label{alg:rbms}
\begin{algorithmic}[1]
\State Split $K=:K_1+K_2$, where $K_1$ has short range, while $K_2$ has long range but is smooth.

\For {$m \text{ in } 1: [T/\tau]$}   
\State Divide $\{1, 2, \ldots, N=pn\}$ into $n$ batches randomly.
     \For {each batch  $\mathcal{C}_q$} 
     \State Update $X^i$'s ($i\in \mathcal{C}_q$) by solving for $t\in [t_{m-1}, t_m)$
     \begin{gather}
     \begin{split}
             dX^i= &\Big( b(X^i)+\alpha_N\sum_{j: j\neq i}K_1(X^i-X^j)\\
                       &+\frac{\alpha_N(N-1)}{p-1}\sum_{j\in\mathcal{C}_q,j\neq i}K_2(X^i-X^j)\Big)\,dt+\sigma\, dW^i,
      \end{split}
      \end{gather}
      or 
     \begin{gather}\label{eq:rbm2ndorder}
     \begin{split}
             dX^i=\,&V^i\,dt ,\\
             dV^i=\,&\Big[ b(X^i)+\alpha_N\sum_{j: j\neq i}K_1(X^i-X^j)-\gamma V^i\Big]\,dt\\
            &+\frac{\alpha_N(N-1)}{p-1}\sum_{j\in\mathcal{C}_q,j\neq i}K_2(X^i-X^j)\,dt+\sigma\, dW^i.
      \end{split}
      \end{gather}
      \EndFor
 \EndFor
\end{algorithmic}
\end{algorithm}

\section{A strong convergence analysis}\label{sec:esterr}

In this section, we perform a strong convergence analysis of RBM for the second order systems \eqref{eq:Nbody2nd} in the mean field regime ( i.e., $\alpha_N=1/(N-1)$). The proof largely makes use of the underlying contraction property for the underdamped Langevin equations (\cite{mattingly2002ergodicity,eberle2019couplings}). Note that due to the degeneracy of the noise terms, the contraction should be proved by suitably chosen variables and Lyapunov functions.

 For the notational convenience, we denote $(X_i, V_i)$ to be the solutions given by \eqref{eq:Nbody2nd}. We denote $(\tilde{X}_i, \tilde{V}_i)$ the solutions given by the RBM process \eqref{eq:RBM2nd}. 
We again use the synchronization coupling as in \cite{jin2020random,jin2020convergence}:
\begin{gather}\label{eq:coupling}
X^i(0)=\tilde{X}^i(0) \sim \mu_0,~~W^i=\tilde{W}^i.
\end{gather}

Let $\mathcal{C}_q^{(k)}$ be the batches at $t_k$ where $1\le q\le n$. 
Define
\begin{gather}
\mathcal{C}^{(k)}:=\{\mathcal{C}_1^{(k)}, \cdots, \mathcal{C}_n^{(k)}\},
\end{gather}
to be the random division of batches at $t_{k}$.
By the Kolmogorov extension theorem \cite{durrett2010}, there exists a probability space $(\Omega, \mathcal{F}, \bbP)$ such that the random variables $\{X_0^{i}, W^i, \mathcal{C}^{(k)}: 1\le i\le N, k\ge 0\}$ are all defined on this probability space and are independent.
Then, $\E$ corresponds to the integration on $\Omega$ with respect to the probability measure $\bbP$. Introduce the $L^2(\cdot)$ norm
\begin{gather}
\|v\|=\sqrt{\mathbb{E}|v|^2}.
\end{gather}

Introduce the filtration $\{\mathcal{F}_{k}\}_{k\ge 0}$ by
\begin{gather}
\mathcal{F}_{k-1}=\sigma(X_0^i, W^i(t), \mathcal{C}^{(j)}; t\le t_{k-1}, j\le k-1).
\end{gather}
Thus, $\mathcal{F}_{k-1}$ is the $\sigma$-algebra generated by the initial values $X_0^i$, $W^i(t)$, and $\mathcal{C}^{(j)}$ for all $i=1,\ldots, N$, $t\le t_{k-1}$ and $j\le k-1$. Clearly, $\mathcal{F}_{k-1}$ contains the information on how batches are constructed for $t\in [t_{k-1}, t_k)$.

For finite time interval, the convergence of RBM is straightforward, as shown below in Proposition \ref{pro:finiteconv}. The proof is similar to that of the results in \cite{jin2020random,jin2020convergence}, and we omit.
\begin{proposition}\label{pro:finiteconv}
Let $b(\cdot)$ be Lipschitz continuous, and $|b|, |\nabla b|$ have polynomial growth.  The interaction kernel $K$ is Lipschitz continuous. Then,
\begin{gather}
\sup_{t\in [0, T]}\sqrt{\E|\tilde{X}^1-X^1|^2+\E|\tilde{V}^1-V^1|^2}
\le C(T)\sqrt{\frac{\tau}{p-1}+\tau^2},
\end{gather}
where $C(T)$ is independent of $N$.
\end{proposition}

Below,  we consider the error estimate for long time. This is important if one uses RBM as a sampling method for the invariant measure of \eqref{eq:Nbody2nd}. To establish the error estimate, we need some technical assumptions that may seem restrictive for practical use. 
The following conditions will give certain contraction property for the second order systems.
\begin{assumption}\label{ass:convexity}
Suppose $b=-\nabla U$ for some $U$ that is bounded below $\inf_x U(x)>-\infty$, and there exist $\lambda_M\ge \lambda_m>0$ such that the eigenvalues of $H:=\nabla^2U$ satisfy
\[
\lambda_m \le \lambda_i(x)\le \lambda_M,~\forall~1\le i\le d, x\in \R^d.
\]
The interaction kernel $K$ is bounded and Lipschitz continuous. Moreover, the friction $\gamma$ and the Lipschitz constant $L$ of $K(\cdot)$ satisfy
\begin{gather}
 \gamma>\sqrt{\lambda_M+2L},~~\lambda_m>2L.
\end{gather}
\end{assumption}
\begin{remark}
The assumptions here are a little different from those for first order systems (\cite{jin2020random,jin2020convergence}): (1) $b$ is assumed to be Lipschitz instead of one-sided Lipschitz; (2) we are not assuming the second derivatives of $K$ to be bounded, as there is no white noise in the equations for $X_i$ so trajectories of $X_i$'s are much smoother.
\end{remark}

Under the assumptions above, we are able to establish the following uniform strong convergence estimate.
\begin{theorem}\label{thm:longtimeconv}
Under Assumption \ref{ass:convexity} and the coupling \eqref{eq:coupling}, 
the solutions to \eqref{eq:Nbody2nd} and \eqref{eq:RBM2nd} satisfy
\begin{gather}
\sup_{t\ge 0}\sqrt{\E|\tilde{X}^1(t)-X^1(t)|^2+\E|\tilde{V}^1(t)-V^1(t)|^2}
\le C\sqrt{\frac{\tau}{p-1}+\tau^2},
\end{gather}
where the constant $C$ does not depend on $p$ and $N$.
\end{theorem}

We give some useful lemmas for our use later.
Denote
\begin{gather}
\mathfrak{X}=(X^1,\cdots, X^N),~~~
\tilde{\mathfrak{X}}=(\tilde{X}^1,\cdots, \tilde{X}^N),
\end{gather}
and introduce the random variables $I_{ij}$ to indicate whether the two particles are in the same batch or not
\begin{gather}
I_{ij}=
\begin{cases}
1 & \exists~\mathcal{C}_q, i,j\in \mathcal{C}_q\\
0 & \text{otherwise}
\end{cases},
1\le i, j\le N.
\end{gather}

Lemma \ref{lmm:estimatesofindicator}-Lemma \ref{lmm:normofrandomsum} below are in \cite{jin2020convergence}, and we omit their proofs.
\begin{lemma}\label{lmm:estimatesofindicator}
For $i\neq j$, it holds that
\begin{gather}
\mathbb{E}I_{ij}=\frac{p-1}{N-1},
\end{gather}
and for distinct $i,j,\ell$, it holds that
\begin{gather}
\mathbb{P}(I_{ij}I_{i\ell}=1)=\E I_{ij}I_{i\ell}=\frac{(p-1)(p-2)}{(N-1)(N-2)}.
\end{gather}
\end{lemma}

For given $\mathtt{x}:=(x^1, \ldots, x^N)\in \mathbb{R}^{Nd}$, we introduce the error of the interacting force
for the $i$th particle.
\begin{gather}
\chi_i(\mathtt{x}):=\frac{1}{p-1}\sum_{j\in\mathcal{C}}K(x^i-x^j)
-\frac{1}{N-1}\sum_{j:j\neq i}K(x^i-x^j).
\end{gather}
Here, $\mathcal{C}$ is the random batch that contains $i$ in a random division of the batches.

\begin{lemma}\label{lmm:consistency}
It holds that
    \begin{gather}
        \mathbb{E}\chi_i(\mathtt{x})=0.
    \end{gather}
    Moreover, the second moment is given by
\begin{gather}\label{eq:var1}
        \mathbb{E}|\chi_i(\mathtt{x})|^2 = 
        \left(\frac{1}{p-1}-\frac{1}{N-1}\right)\Lambda_i(\mathtt{x}),
\end{gather}
    where
\begin{gather}\label{eq:Lambda}
        \Lambda_i(\mathtt{x}):=\frac{1}{N-2}
        \sum_{j: j\neq i}\Big|  K(x^i-x^j)-\frac{1}{N-1}
        \sum_{\ell: \ell\neq i}K(x^i-x^\ell)  \Big|^2.
\end{gather}
\end{lemma}

\begin{lemma}\label{lmm:normofrandomsum}
Fix $i\in\{1,\ldots, N\}$. Let $\mathcal{C}_{\theta}$ be the random batch of size $p$ that contains $i$ in the random division. 
Let $Y_j$ ($1\le j\le N$) be $N$ random variables (or random vectors) that are independent of $\mathcal{C}_{\theta}$. Then, for $p\ge 2$,
\begin{gather}
\left\|\frac{1}{p-1}\sum_{j\in\mathcal{C}_{\theta},j\neq i}Y_j\right\|\le 
\left(\frac{1}{N-1}\sum_{j: j\neq i}\|Y_j\|^2\right)^{1/2}.
\end{gather}
\end{lemma}

Below, we establish some moment estimates so that we can establish the stability for the Random Batch Method and thus prove convergence.
\begin{lemma}\label{lmm:conditionalest}
Under Assumption \ref{ass:convexity}, it holds for $q\ge 1$ that
\begin{gather}\label{eq:uniformmoments}
\sup_{t>0}\Big(\E(|X^i(t)|^q+|V^i(t)|^q)+\mathbb{E}(|\tilde{X}^i(t)|^q+|\tilde{V}^i(t)|^q)\Big) \le C_q.
\end{gather}
Besides, for any $k>0$ and $q\ge 2$,
\begin{gather}\label{eq:conditionmoment}
\sup_{t\in [t_{k-1}, t_k)}\left|\mathbb{E}(|\tilde{X}^i(t)|^q+|\tilde{V}^i(t)|^q |\mathcal{F}_{k-1})\right|
\le C(1+|\tilde{X}^i(t_{k-1})|^q+|\tilde{V}^i(t_{k-1})|^q)
\end{gather}
holds almost surely.
\end{lemma}
\begin{proof}
Here, we show the moment bounds for $(\tilde{X}^i, \tilde{V}^i)$ only, as the estimates for the moments of $(X^i, V^i)$ are similar (and easier).

Following \cite[section 3]{mattingly2002ergodicity}, we consider the Lyapunov function
\[
\ell(\tilde{X}^i, \tilde{V}^i):=\frac{1}{2}(|\tilde{X}^i|^2+|\tilde{X}^i+\alpha \tilde{V}^i|^2)+\alpha^2 U(\tilde{X}^i).
\]
By Assumption \ref{ass:convexity}, we can assume without loss of generality that
\[
\inf_x U(x)=0.
\]
Due to Assumption \ref{ass:convexity}, the second moments of $\tilde{X}^i$ and $\tilde{V}^i$ can be controlled easily by this Lyapunov function.

Then, by It\^o's formula, for any $r\ge1$, and $t\in [t_{k-1}, t_k)$,
\[
\frac{d}{dt}\E\left[ [\ell(\tilde{X}^i, \tilde{V}^i)]^r|\mathcal{F}_{k-1}\right]
=\E[\cL[\ell(\tilde{X}^i, \tilde{V}^i)]^r |\mathcal{F}_{k-1}],
\]
where $\cL$ is the generator for the SDE \eqref{eq:RBM2nd} given by
\begin{multline}
\cL=\sum_{i=1}^N v^i\cdot \nabla_{x^i}+\sum_{i=1}^N\left(-\nabla U(x^i)+\frac{1}{p-1}\sum_{j\in \mathcal{C}_{\theta},j\neq i}K(x^i-x^j)
-\gamma v^i\right)\cdot\nabla_{v^i}\\
+\frac{1}{2}\sum_{i=1}^N\sigma^2\Delta_{v_i}.
\end{multline}
Note that $x^i\in \R^d$ and $v^i\in \R^d$.
Direct computation reveals that
\[
\cL [\ell(x^i,v^i)]^r=r[\ell(x^i,v^i)]^{r-1}\cL\ell(x^i,v^i)
+\frac{\sigma^2}{2}r(r-1)[\ell(x^i,v^i)]^{r-2}|\nabla_{v^i}\ell|^2.
\]

Clearly, $|\nabla_{v^i}\ell|^2\le \chi(\alpha)\ell(x^i,v^i)$ for some number $\chi$ depending on $\alpha$. That means
\[
\frac{\sigma^2}{2}r(r-1)[\ell(x^i,v^i)]^{r-2}|\nabla_{v^i}\ell|^2\le C[\ell(x^i,v^i)]^{r-1}.
\]
The power $r-1$ indicates that this term can be controlled without difficulty. Moreover,
\begin{align*}
\cL\ell(x^i, v^i)=&v^i\cdot\left(x^i+(x^i+\alpha v^i)+\alpha^2\nabla U \right)\\
&+\left[-\nabla U(x^i)-\gamma v^i+\frac{1}{p-1}\sum_{j\in \mathcal{C}_{\theta},j\neq i}K(x^i-x^j)\right]\cdot \alpha(x^i+\alpha v^i)+\frac{1}{2}\sigma^2 \alpha^2 d\\
=&-\alpha x^i\cdot\nabla U(x^i)+(\alpha-\alpha^2\gamma)|v^i|^2+(2-\alpha \gamma)x^i\cdot v^i \\
&+\frac{\alpha}{p-1}x^i\cdot\sum_{j\in \mathcal{C},j\neq i}K(x^i-x^j)
+\frac{\alpha^2}{p-1}v^i\cdot\sum_{j\in \mathcal{C},j\neq i}K(x^i-x^j)
+\frac{1}{2}\sigma^2 \alpha^2 d.
\end{align*}
Taking $\alpha=2/\gamma$, one finds that
\[
\cL\ell(x^i, v^i)\le -\frac{2}{\gamma} \lambda_m |x^i|^2
-\frac{2}{\gamma}|v^i|^2
+C(|x^i|+|v^i|+1)\le -\beta \ell+C,
\]
for some $\beta>0$.
Hence,
\[
\frac{d}{dt}\E\left[ [\ell(\tilde{X}^i, \tilde{V}^i)]^r|\mathcal{F}_{k-1}\right]
\le -r\beta\E\left[ [\ell(\tilde{X}^i, \tilde{V}^i)]^r|\mathcal{F}_{k-1}\right]
+C\E\left[ [\ell(\tilde{X}^i, \tilde{V}^i)]^{r-1}|\mathcal{F}_{k-1}\right].
\]
Using the fact that $U(x^i)\le C(1+|x^i|^2)$ (since $\|\nabla^2U\|_2\le \lambda_M$), \eqref{eq:conditionmoment} then follows easily. 

Similarly,  taking full expectation leads to
\[
\frac{d}{dt}\E\left[ [\ell(\tilde{X}^i, \tilde{V}^i)]^r\right]
\le -r\beta\E\left[ [\ell(\tilde{X}^i, \tilde{V}^i)]^r\right]
+C\E\left[ [\ell(\tilde{X}^i, \tilde{V}^i)]^{r-1}\right].
\]
The moments control \eqref{eq:uniformmoments} then follows.
\end{proof}

Next, we introduce some notations for better presentation.
First of all, due to the degeneracy of the white noise in the equations for $X_i$'s, the generator of the underdamped Langevin does not have the ellipticity. The proof of the ergodicity for the corresponding Fokker-Planck equation relies on the hypocoercivity \cite{villani2006hypocoercivity,dolbeault2015hypocoercivity} and one needs to use the transport term to compensate the degeneracy. In terms of particle formulation (SDEs), one may consider the following variables to compensate the degeneracy (see \cite[section 3]{mattingly2002ergodicity} or \cite{eberle2019couplings}):
\begin{gather}
Z^i:=\tilde{X}^i-X^i,~~\hat{Z}^i:=\tilde{X}^i-X^i+\alpha 
(\tilde{V}^i-V^i).
\end{gather}
Moreover, we denote
\begin{gather}
\delta K_{ij}(t):=K(\tilde{X}^i(t)-\tilde{X}^j(t))-K(X^i(t)-X^j(t)).
\end{gather}
Using this notation, we can conveniently write
\begin{gather}
\begin{split}
\frac{1}{p-1}\sum_{j\in \mathcal{C}_\theta j\neq i}K(\tilde{X}^i-\tilde{X}^j)-\frac{1}{N-1}\sum_{j\neq i}K(X^i-X^j)
&=\frac{1}{N-1}\sum_{j\neq i}\delta K_{ij}+\chi_{i}(\tilde{\mathfrak{X}})\\
&=\frac{1}{p-1}\sum_{j\in \mathcal{C}_{\theta},j\neq i}\delta K_{ij}+\chi_i(\mathfrak{X}).
\end{split}
\end{gather}

By the definition of $\hat{Z}^i$,
\[
\frac{d}{dt}\hat{Z}^i=(1-\alpha \gamma)(\tilde{V}^i(t)-V^i(t))
-\alpha\Big(\nabla U(\tilde{X}^i)-\nabla U(X^i)\Big)
+\frac{1}{p-1}\sum_{j\in \mathcal{C}_{\theta},j\neq i}\delta K_{ij}
+\chi_i(\mathfrak{X}).
\]

\begin{lemma}\label{lmm:Zestimate}
Suppose Assumption \ref{ass:convexity} holds. For $t\in [t_{k-1}, t_k)$,
\begin{gather}
\|Z^i(t)-Z^i(t_{k-1})\|+\|\hat{Z}^i(t)-\hat{Z}^i(t_{k-1})\|\le C \tau.
\end{gather}
Also, almost surely, it holds that for $\tau\le 1$,
\begin{gather}
 |Z^i(t)|+|\hat{Z}^i(t)| \le ( |Z^i(t_{k-1})| +  |\hat{Z}^i(t_{k-1})| )(1+C\tau)+C\tau.
\end{gather}
\end{lemma}
\begin{proof}
The first part is an easy consequence of the moment control in \eqref{eq:uniformmoments}. 
Direct computation shows that
\[
\begin{split}
 & \frac{d}{dt}|Z^i|\le \frac{1}{\alpha}|\hat{Z}^i-Z^i|,\\
 & \frac{d}{dt}|\hat{Z}^i|
 \le   \Big|\frac{(1-\alpha\gamma)(\hat{Z}^i-Z^i)}{\alpha}+\alpha
 \Big(-\nabla U(\tilde{X}^i)+\nabla U(X^i)
 +\frac{1}{p-1}\sum_{j\in \mathcal{C}_{\theta},j\neq i}\delta K_{ij}
+\chi_i(\mathfrak{X})\Big)\Big|.
 \end{split}
\]
Hence, one has
\[
\frac{d}{dt}(|Z^i|+|\hat{Z}^i|)
\le C(|Z^i|+|\hat{Z}^i|)+C.
\]
The claim then follows.
\end{proof}

\begin{proof}[Proof of Theorem \ref{thm:longtimeconv}]
We aim to estimate how the quantity evolves
\begin{gather}
\begin{split}
u(t) &:=\frac{1}{N}\sum_{i=1}^N (\E|\tilde{X}^i-X^i|^2
+\E|V^i-\tilde{V}^i|^2) \\
&= \E|\tilde{X}^1-X^1|^2+\E|\tilde{V}^1-V^1|^2.
\end{split}
\end{gather}
Direct estimation of this quantity is not easy. As mentioned above already, we consider the following motivated by \cite{eberle2019couplings}:
\begin{gather}
\begin{split}
J(t) &:=\frac{1}{2}(\E |Z^1|^2+\E|\hat{Z}^1|^2)\\
&=\frac{1}{2}\left[\E|\tilde{X}^1-X^1|^2+\E|\tilde{X}^1-X^1+\alpha 
(\tilde{V}^1-V^1)|^2\right],
\end{split}
\end{gather}
where $\alpha$ is to be determined later. Then, $J$ is equivalent to $u$ but it can be treated more easily as we shall see below.

\subsubsection*{Step 1. Contraction}

Recall
\[
\frac{1}{p-1}\sum_{j\in \mathcal{C}_{\theta}, j\neq 1}K(\tilde{X}^1-\tilde{X}^j)-\frac{1}{N-1}\sum_{j\neq 1}K(X^1-X^j)
=\frac{1}{N-1}\sum_{j\neq 1}\delta K_{1j}+\chi_{1}(\tilde{\mathfrak{X}}),
\]
where $\mathcal{C}_{\theta}$ again is the random batch that contains particle $1$ and define
\begin{gather}\label{eq:B}
B(\tilde{X}^1, X^1):=\int_0^1\nabla^2U(s \tilde{X}^1+(1-s)X^1)\,ds.
\end{gather}
Direct computation yields
\begin{multline*}
\frac{d}{dt}J=\alpha^{-1}\E Z^1\cdot(\hat{Z}^1-Z^1)
+\alpha\E\hat{Z}^1\cdot\Bigg(\alpha^{-2}(\hat{Z}^1-Z^1)-B\cdot Z^1\\
+\frac{1}{N-1}\sum_{j\neq 1}\delta K_{1j}+\chi_{1}(\tilde{\mathfrak{X}})
-\frac{\gamma}{\alpha}(\hat{Z}^1-Z^1)\Bigg).
\end{multline*}

By symmetry, $\|Z^j\|=\|Z^1\|$, and thus
\[
\E \hat{Z}^1\cdot \frac{1}{N-1}\sum_{j\neq 1}\delta K_{1j}
\le L (\|\hat{Z}^1\| \|Z^1\|+\sum_{j\neq 1}\|\hat{Z}^1\| \|Z^j\|) \le 2L J(t).
\]

Then,
\begin{multline}\label{eq:dotJinter1}
\dot{J}
\le -\E\left( [Z^1, \hat{Z}^1]^T
\left[\begin{array}{cc}
\alpha^{-1}I_d & \frac{1}{2}(\alpha B-\gamma I_d)\\
\frac{1}{2}(\alpha B-\gamma I_d) & \gamma-\alpha^{-1}I_d
\end{array}
\right]
\left[\begin{array}{c}
Z^1\\
\hat{Z}^1
\end{array}
\right]\right) \\
+2\alpha  L J(t)+\alpha \E \hat{Z}^1\cdot \chi_1(\tilde{\mathfrak{X}}).
\end{multline}

Let the eigenvalues of $B$ be $\tilde{\lambda}_i$, which are
bounded below and above by $\lambda_m$, $\lambda_M$ respectively.
The eigenvalues of the matrix in \eqref{eq:dotJinter1} are given by
\[
\mu_{i,\pm}=\frac{1}{2}(\gamma\pm \sqrt{\gamma^2+4\alpha^{-1}
(\alpha^{-1}-\gamma)+(\alpha\tilde{\lambda}_i-\gamma)^2}).
\]
Choosing $\alpha=2/\gamma$,
so that all the eigenvalues are $
\left\{ \tilde{\lambda}_i/\gamma, \gamma-\tilde{\lambda}_i/\gamma \right\}_{i=1}^d$.
Hence,
\begin{gather*}
\dot{J}\le -\frac{2}{\gamma}\left[\min(\lambda_m-2L, \gamma^2-(\lambda_M+2L)) \right]J(t)+\frac{2}{\gamma}\E\hat{Z}\cdot\chi_1(\tilde{\mathfrak{X}}).
\end{gather*}
Under Assumption \ref{ass:convexity}, the coefficient of $J(t)$ on the right hand side is negative so that the Langevin dynamics has contraction property.

\subsubsection*{Step 2. The local error estimate. }

We now estimate the local error term
$\E\hat{Z}^1\cdot\chi_1(\tilde{\mathfrak{X}}(t))$, which we decompose as
\[
\begin{split}
\E\hat{Z}^1\cdot\chi_1(\tilde{\mathfrak{X}}(t)) &=
\E\hat{Z}^1(t_{k-1})\cdot\chi_1(\tilde{\mathfrak{X}}(t))
+\E[\hat{Z}^1(t)-\hat{Z}^1(t_{k-1})]\cdot\chi_1(\tilde{\mathfrak{X}}(t)) \\
&=: I_1+I_2.
\end{split}
\]
\subsubsection*{Substep 2.1. Estimation of $I_1$}

Using the consistency result Lemma \ref{lmm:consistency},
\[
\E\hat{Z}^1(t_{k-1})\cdot\chi_1(\tilde{\mathfrak{X}}(t_{k-1}))=0,
\]
one has
\[
\begin{split}
I_1&=\E\hat{Z}^1(t_{k-1})\cdot[\chi_1(\tilde{\mathfrak{X}}(t))
-\chi_1(\tilde{\mathfrak{X}}(t_{k-1}))]\\
&=\E\left(\hat{Z}^1(t_{k-1})\cdot\E[\chi_1(\tilde{\mathfrak{X}}(t))-\chi_1(\tilde{\mathfrak{X}}(t_{k-1}))|\mathcal{F}_{k-1}]\right)\\
&\le \|\hat{Z}^1(t_{k-1})\|
\left\|\E[\chi_1(\tilde{\mathfrak{X}}(t))-\chi_1(\tilde{\mathfrak{X}}(t_{k-1}))|\mathcal{F}_{k-1}] \right\|.
\end{split}
\]
Introducing
\begin{gather*}
\begin{split}
& \delta \tilde{K}^{1j}:=K(\tilde{X}^1(t)-\tilde{X}^j(t))
-K(\tilde{X}^1(t_{k-1})-\tilde{X}^j(t_{k-1})),\\
& \delta \tilde{X}^j=\tilde{X}^j(t)-\tilde{X}^j(t_{k-1}),
\end{split}
\end{gather*}
one has
\begin{gather}\label{eq:step21aux1}
\begin{split}
& |\E[\chi_1(\tilde{X}(t))-\chi_1(\tilde{X}(t_{k-1}))|\mathcal{F}_{k-1}]| \\
&\le \frac{1}{p-1}\sum_{j\in \mathcal{C}_{\theta},j\neq 1}
|\E(\delta \tilde{K}^{1j}(t)|\mathcal{F}_{k-1})|
+\frac{1}{N-1}\sum_{j\neq 1}
|\E(\delta \tilde{K}^{1j}(t)|\mathcal{F}_{k-1})|.
\end{split}
\end{gather}

Now, we estimate
\begin{gather*}
\left\| \frac{1}{p-1}\sum_{j\in \mathcal{C}_{\theta},j\neq 1}
|\E(\delta \tilde{K}^{1j}(t)|\mathcal{F}_{k-1})| \right\|
\le L\| \E(\delta \tilde{X}^1|\mathcal{F}_{k-1}) \|
+\frac{L}{p-1} \left\| \sum_{j\in \mathcal{C}_{\theta},j\neq 1}|\E(\delta \tilde{X}^j|\mathcal{F}_{k-1})| \right\|.
\end{gather*}

Since $\delta\tilde{X}^j=\int_{t_{k-1}}^t \tilde{V}^j(s)\,ds$, we find easily that
\[
\begin{split}
|\E(\delta\tilde{X}^j | \mathcal{F}_{k-1})|
&\le \int_{t_{k-1}}^t\E(|\tilde{V}^j(s)| | \mathcal{F}_{k-1})\,ds\\
&\le \int_{t_{k-1}}^s\sqrt{\E(|\tilde{V}^j(s)|^2 | \mathcal{F}_{k-1})}\,ds \\
& \le C\left( \sqrt{1+|\tilde{X}^j(t_{k-1})|^2+|\tilde{V}^j(t_{k-1})|^2}\right)\tau.
\end{split}
\]
Now, since $ \sqrt{1+|\tilde{X}^j(t_{k-1})|^2+|\tilde{V}^j(t_{k-1})|^2}$ is independent of $\mathcal{C}_{\theta}$, applying Lemma \ref{lmm:normofrandomsum}, one has
\[
\left\| \frac{L}{p-1}\sum_{j\in \mathcal{C}_{\theta},j\neq 1}\sqrt{1+|\tilde{X}^j(t_{k-1})|^2+|\tilde{V}^j(t_{k-1})|^2} \right\|
\le C.
\]
The other term in \eqref{eq:step21aux1} is similar, but much simpler.

Hence, we find
\[
I_1\le C\|\hat{Z}^1(t_{k-1})\|\tau\le C\|\hat{Z}(t)\|\tau+C\tau^2.
\]

\subsubsection*{Substep 2.2. Estimate of $I_2$}

Now, we estimate $I_2$.
\[
\begin{split}
I_2&=\E[\hat{Z}^1(t)-\hat{Z}^1(t_{k-1})]\cdot\chi_1(\tilde{\mathfrak{X}})\\
&=\E[\hat{Z}^1(t)-\hat{Z}^1(t_{k-1})]\cdot\chi_1(\mathfrak{X})
+\E[\hat{Z}^1(t)-\hat{Z}^1(t_{k-1})]\cdot[\chi_1(\tilde{\mathfrak{X}})-\chi_1(\mathfrak{X})]\\
&=: I_{21}+I_{22}.
\end{split}
\]

For $I_{21}$, we recall the matrix $B$  defined in \eqref{eq:B}  and its spectral radius by Assumption \ref{ass:convexity} is controlled by
\[
\rho(B(\tilde{X}^1, X^1)) \le \lambda_M.
\]
Then, one has
\begin{gather}\label{eq:diffhatZaux1}
\begin{split}
\hat{Z}^1(t)-\hat{Z}^1(t_{k-1}) = &\int_{t_{k-1}}^t (\alpha^{-1}-\gamma)(\hat{Z}^1-Z^1)ds
-\alpha\int_{t_{k-1}}^t Z^1\cdot B(\tilde{X}^1(s), X^1(s)) ds\\
&+\int_{t_{k-1}}^t \frac{1}{p-1}\sum_{j\in \mathcal{C}_{\theta}, j\neq 1}
\delta K_{1j}\,ds+\int_{t_{k-1}}^t \chi_1(\mathfrak{X}(s))\,ds.
\end{split}
\end{gather}
Noting Lemma \ref{lmm:Zestimate}, $\alpha^{-1}-\gamma=-\gamma/2$, and that $\|\chi_i\|_{\infty}\le 2\|K\|_{\infty}$, one has
\[
\E\left(\int_{t_{k-1}}^t \frac{1}{2}\gamma(\hat{Z}^1-Z^1)ds
-\alpha\int_{t_{k-1}}^t Z^1\cdot B(\tilde{X}^1, X^1) ds\right)\cdot \chi_1(\mathfrak{X}(t))
\le C\sqrt{J(t)} \tau+C\tau^2,
\]
where we used $\E Z^1\cdot B(\tilde{X}^1(s), X^1(s))\cdot\chi_1(\mathfrak{X}(t))
\le C\lambda_M\|Z^1(s)\| \le C\|Z^1(t)\|+C\tau$.

For the third term in \eqref{eq:diffhatZaux1} dotted with $\chi_i$, one has 
\[
\E \left[ \Big( \int_{t_{k-1}}^t \frac{1}{p-1}\sum_{j\in \mathcal{C}_{\theta}, j\neq 1} \delta K_{1j}\,ds \Big) \cdot \chi_1(\mathfrak{X}(t))\right]
\le \|\chi_1\|_{\infty} \sup_{s\in [t_{k-1}, t_k]} \left\| \frac{1}{p-1}\sum_{j\in \mathcal{C}_{\theta}, j\neq 1} \delta K_{1j}(s) \right\|  \tau. 
\]

Clearly,
\begin{gather}\label{eq:pfaux1}
\left\| \frac{1}{p-1}\sum_{j\in \mathcal{C}_{\theta}, j\neq 1} \delta K_{1j}(s) \right\|
\le L\left(\|Z^1(s)\|+  \Big\| \frac{1}{p-1}\sum_{j\in \mathcal{C}_{\theta}, j\neq 1}|Z^j(s)| \Big\| \right).
\end{gather}

By Lemma \ref{lmm:Zestimate}, one has almost surely that
\[
|Z^j(s)|\le |Z^j(t_{k-1})|+C\tau,
\]
and thus
\[
\Big\| \frac{1}{p-1}\sum_{j\in \mathcal{C}_{\theta}, j\neq 1}|Z^j(s)| \Big\|
\le \Big\| \frac{1}{p-1}\sum_{j\in \mathcal{C}_{\theta}, j\neq 1}|Z^j(t_{k-1})| \Big\|
+C\tau.
\]
Since $\{|Z^j(t_{k-1})|\}$'s are independent of $\mathcal{C}_{\theta}$, 
Lemma \ref{lmm:normofrandomsum} then gives
\[
\begin{split}
\left\| \frac{1}{p-1}\sum_{j\in \mathcal{C}_{\theta}, j\neq 1}|Z^j(t_{k-1})| \right\|
 &\le \Big(\frac{1}{N-1}\sum_{j\neq i}\|Z^j(t_{k-1})\|^2 \Big)^{1/2}\\
&=\|Z^1(t_{k-1})\|.
\end{split}
\]
Hence, one in fact has
\begin{gather}\label{eq:pfaux2}
\begin{split}
\left\| \frac{1}{p-1}\sum_{j\in \mathcal{C}_{\theta}, j\neq 1} \delta K_{1j}(s) \right\|
 &\le 2L\|Z^1(t_{k-1})\|+ C\tau  \\
 &\le 2L\|Z^1(t)\|+C\tau.
\end{split}
\end{gather}

The fourth term can be easily controlled using Lemma \ref{lmm:consistency} so that
\[
\E \int_{t_{k-1}}^t \chi_1(\mathfrak{X}(s))\cdot \chi_1(\mathfrak{X}(t))\,ds
\le \Big(\frac{1}{p-1}-\frac{1}{N-1}\Big)\|\Lambda_1\|_{\infty}\tau.
\]

Now, we move to the estimate of $I_{22}$ term, which is much easier. 
In fact,
\[
I_{22}=\E[\hat{Z}^1(t)-\hat{Z}^1(t_{k-1})]\cdot[\chi_1(\tilde{\mathfrak{X}})-\chi_1(\mathfrak{X})]
\le C\tau\|\chi_1(\tilde{\mathfrak{X}})-\chi_1(\mathfrak{X})\|,
\]
where $\|\hat{Z}^1(t)-\hat{Z}^1(t_{k-1}) \|\le C\tau$ by Lemma \ref{lmm:Zestimate}.

Now, by the definition of $\chi_i$,
\[
\|\chi_1(\tilde{\mathfrak{X}})-\chi_1(\mathfrak{X})\|
\le \left\| \frac{1}{p-1}\sum_{j\in \mathcal{C}_{\theta}, j\neq 1} \delta K_{1j}(s) \right\|
+\left\| \frac{1}{N-1}\sum_{j: j\neq 1} \delta K_{1j}(s) \right\|.
\]
The first term has been estimated in \eqref{eq:pfaux2}. The second term is easily bounded with the same bound as in \eqref{eq:pfaux2}. Hence, $I_2$ is controlled as
\[
I_2 \le C\sqrt{J(t)}\tau+C\tau^2+\frac{1}{p-1}\|\Lambda_1\|_{\infty}\tau.
\]

\subsubsection*{Step 3: final error estimate}

Combining all the estimates above, one has
\begin{gather}
\dot{J}\le -\frac{2}{\gamma}\left[\min(\lambda_m-2L, \gamma^2-(\lambda_M+2L)) \right]J(t)
+\frac{C}{\gamma}\left(\sqrt{J(t)}\tau+\tau^2+\frac{1}{p-1}\tau 
\right).
\end{gather}
Gr\"onwall's inequality then gives the desired result.

\end{proof}

\section{Applications to molecular dynamics simulations}\label{sec:md}
In this section, we discuss possible applications of RBM with splitting to molecular dynamics simulations. 

Molecular dynamics refers to computer simulation of atoms and molecules
to compute the statistics of the distribution and investigate the properties of solids and fluids \cite{ciccotti1987simulation,frenkel2001understanding}.
Consider $N$ ``molecules'' (each might be a model for a real molecule or a numerical molecule that is a packet of many real molecules) that interact with each other:
\begin{gather}\label{eq:md1}
\begin{split}
& dX^i=V^i\,dt,\\
& dV^i=\Big[-\sum_{j:j\neq i}\nabla \phi(X^i-X^j)\Big]\,dt
+d\xi^i.
\end{split}
\end{gather}
Here, $\phi(\cdot)$ is the interaction potential and $d\xi^i$ means some other possible terms that change the momentum, which we will discuss below. Typical examples include the Coulomb potentials 
\[
\phi(x) = \frac{q_i q_j}{r},
\]
where $q_i$ is the charge for the $i$th particle and $r=|x|$, and the Lennard-Jones potential 
\[
\phi(x)=4\left(\frac{1}{r^{12}}-\frac{1}{r^6}\right).
\]
Between ions, both types of potential exist and between charge-neutral molecules, the Lennard-Jones potential might be the main force (the Lennard-Jones interaction intrinsically also arises from the interactions between charges, so these two types are in fact both electromagnetic forces).  To model the solids or fluids with large volume, one often uses a box with length $L$, equipped with the periodic conditions for the simulations.

\subsection{Coupling with the heat bath}\label{subsec:heatbath}

To model the interaction between the molecules with the heat bath, one may consider some thermostats so that the temperature of the system can be controlled at a given value. Typical thermostats include the Andersen thermostat, the Langevin thermostat and the Nos\'e-Hoover thermostat \cite{frenkel2001understanding}.

In the Andersen thermostat \cite[section 6.1.1]{frenkel2001understanding}, one does the simulation for 
\[
d\xi^i=0,
\] 
but a particle can collide with the heat bath each time. In particular, assume the collision frequency is $\nu$, so in a duration of time $t\ll 1$ the chance that a collision has happened is given by the exponential distribution
\[
1-\exp(-\nu t)\approx \nu t,~~t\ll 1.
\]
If a collision happens, the new velocity is then sampled from the Maxwellian distribution with temperature $T$ (i.e., the normal distribution $\mathcal{N}(0, T)$). 

Since the potential $\phi(x)$ is often singular at $x=0$, we need to split the interaction kernel (or the potential) for the simulation and apply RBM for the long-range but smooth part. Hence, Algorithm \ref{alg:rbms} can be applied when evolving the dynamics. 
Here is some subtlety.  If we discretize \eqref{eq:rbm2ndorder} using some second order integration methods like the Verlet method \cite[section 4.3.1]{frenkel2001understanding}, we need to evaluate the forces at $t_k^-$ and $t_k^+$. The force at $t_k^-$ corresponds to the batches for $[t_{k-1}, t_k)$ while the force at $t_k^+$ corresponds to batches for $[t_k, t_{k+1})$. This is not quite efficient as one needs to evaluate the force at $t_k$ twice, so for practical implementation, we evaluate the forces only at $t_k^+$. Then, in Verlet, the velocity is updated using
\begin{gather}\label{eq:velocityupdateverlet}
V_{k+1}^i=V_k^i+\frac{1}{2}[F_{k^+}^i+F_{(k+1)^+}^i]\tau.
\end{gather}
This of course is not a discretization of \eqref{eq:rbm2ndorder} any more.
However, it is expected that there is no significant difference. In fact, it is known that the Verlet scheme is equivalent to the leapfrog scheme (eqns (4.3.1)--(4.3.2) in \cite[section 4.3.1]{frenkel2001understanding}) and \eqref{eq:velocityupdateverlet} corresponds to applying RBM to the leapfrog scheme.
We will call the corresponding algorithm ``Andersen-RBM", which is shown in Algorithm \ref{alg:adr}.

\begin{algorithm}[H]
\caption{Andersen-RBM}
\label{alg:adr}
\begin{algorithmic}[1]
\State Split $K=:K_1+K_2$: $K_1$ has short range, while $K_2$ has long range but is regular.
\State Sample $X^i, V^i$ for all $i$.
\State Obtain a set of random batches. For each $i$, find the batch $\mathcal{C}$ where $i$ lives, and compute:
     \begin{gather}\label{eq:mdrandomforce}
     F^i=\sum_{j: j\neq i}K_1(X^i-X^j)+\frac{N-1}{p-1}\sum_{j\in\mathcal{C},j\neq i}K_2(X^i-X^j).
     \end{gather}

\For {$m \text{ in } 1: [T/\tau]$}   
      \State For each $i$, generate $\zeta_i\sim U(0, 1)$. If $\zeta_i\le 1-\exp(-\nu \Delta t)$, sample $V_i\sim \mathcal{N}(0, T)$.
     \State Update the positions:
     \begin{gather}
     X^i \leftarrow X^i+V^i\tau+\frac{1}{2}F^i\tau^2.
      \end{gather}
      \State Set $F_o^i\leftarrow F^i$ for all $i$.
      \State Obtain a new set of random batches, and compute the forces for all particles $i\in\{1,\cdots, N\}$ as in \eqref{eq:mdrandomforce}.
      \State Update the velocities for all $i$:
      \[
       V^i \leftarrow V^i+\frac{1}{2}[F_o^i+F^i]\tau.
      \]
 \EndFor
\end{algorithmic}
\end{algorithm}

In the underdamped Langevin dynamics, one chooses 
\[
d\xi^i=-\gamma V^i\,dt+\sqrt{\frac{2\gamma}{\beta}}\,dW^i,
\] 
so that the ``fluctuation-dissipation relation'' is satisfied and the system will evolve to the equilibrium with the correct temperature. 
It is well-known that the invariant measure of such systems is given by the  Gibbs distribution
\[
\pi(\mathtt{x}, \mathtt{v}) \propto \exp\left(-\beta(\frac{1}{2}\sum_{i=1}^N |v^i|^2+E(\mathtt{x}))\right),
\]
where $\mathtt{x}=(x^1,\cdots, x^N)\in \R^{Nd}$ and $\mathtt{v}=(v^1,\cdots, v^N)\in \R^{Nd}$.
 Algorithm \ref{alg:rbms} can be applied directly for Langevin dynamics and one possible way to discretize \eqref{eq:rbm2ndorder} is the ``BAOAB''  splitting scheme proposed in \cite{leimkuhler2013robust,leimkuhler2013rational}:
\begin{gather}
\begin{split}\label{baoba}
& V_{k+\frac{1}{2}}=V_{k}+\frac{1}{2} F_k \tau,\ \ (B)\\
& X_{k+\frac{1}{2}}=X_{k}+\frac{1}{2}V_{k+\frac{1}{2}}\tau,\ \ (A)\\
& \hat{V}_{k+\frac{1}{2}}=c_1V_{k+\frac{1}{2}}+c_2R_{k+1},\ \ (O)\\
& X_{k+1}=X_{k+\frac{1}{2}}+\frac{1}{2}\hat{V}_{k+\frac{1}{2}}\tau,\ \ (A)\\
& V_{k+1}=\hat{V}_{k+\frac{1}{2}}+\frac{1}{2}F_{k+1}\tau,\ \ (B)
\end{split}
\end{gather}
where $F_k$ means the force computed at time $t_k=k\tau$.
The coefficients $c_1=e^{-\gamma\tau}$, $c_2=\sqrt{(1-c_1^2)/\beta}$, where $\gamma$ is the friction coefficient.
Clearly, for second order schemes like this ``BAOAB'' splitting scheme,
the forces $F_k$ should be computed by applying RBM on the $K_2$ part. Again, there can be two possible forces at $t_k$ depending on whether using the batches for $[t_{k-1}, t_k)$ or the batches for $[t_k, t_{k+1})$. Similar as in Algorithm \ref{alg:adr}, we use the forces at $t_k^+$ uniformly. The resulted algorithm is similar to  Algorithm \ref{alg:adr} so we omit it. The resulted discretized scheme will be called ``Langevin-RBM", which does not correspond to the discretization of \eqref{eq:rbm2ndorder} directly, but we believe there is no significant difference.

The Nos\'e-Hoover thermostat (\cite[section 6.1.2]{frenkel2001understanding}) is better behaved for keeping the temperature around the desired value. This is even more desirable when we apply RBM, because RBM can potentially increase temperature of the system by $\sim \tau/p$ due to the variance in the force computation. 
This numerical heating is not good if one wants to obtain some accurate results. In this work, our goal is to validate RBM without asking for very high accuracy, so we only investigate the Andersen thermostat and the underdamped Langevin thermostat later in section \ref{sec:numerics} below. We try to keep the temperature correct by adjusting the friction coefficient and by using decreasing step sizes. In fact, our results in section \ref{sec:numerics} are acceptable using these simple strategies.
If more accurate simulation is needed, the Nos\'e-Hoover thermostat can be considered and will be done in future work.

\subsection{Discussion}\label{subsec:dis}

Below, we first discuss the benefits of RBM in the molecular regime and the choice of the splitting $K=K_1+K_2$. Suppose that $r_0$ is the effective range of $K_1$ (i.e., when $|x|\gg r_0$, the effects of $K_1$ can be neglected).  Hence, to enjoy the benefits of RBM, we may desire to choose $r_0$ so that there are only $\mathcal{O}(1)$ particles in the ball $B(X_i, r_0)$ centered at a typical particle $X_i$.

\begin{itemize}
\item 
For kernels whose range cover effectively only a few particles  (like Lennard-Jones fluid with low density), one can pick $r_0$ large enough. In this case, by the fast decay of the potential, $\sum_{j: j\neq i}K_{2}$ is negligible. Using the random approximation $\frac{N-1}{p-1}\sum_{j\in\mathscr{C}}K_{2}$ is not quite necessary.   For short-range potentials with  $\mathcal{O}(1)$ density, though one may make use of the rapid decay of the potential to make the full simulation cheaper, RBM with splitting can still speed up the simulation for such cases as the batch size $p$ can be smaller than the effective number of neighbors  (see Remark \ref{rmk:ljcut} for Lennard-Jones potentials).

\item 
If the range of $K$ is comparable to the size of the simulation domain (like the Lennard-Jones fluid with period box length $L=1$ and long range interactions like the Coulomb potentials) or the density is not low, then each particle can feel the effects from a significant number of other particles. In this case, we pick $r_0$ small so that $B(X_i, r_0)$ contains $\mathcal{O}(1)$ particles.
RBM can speed up computation per iteration. Moreover, we also expect that $\sum_{j: j\neq i}K_{2}$ and $\frac{N-1}{p-1}\sum_{j\in\mathscr{C}, j\neq i}K_{2}$ will be comparable. The time step needed for RBM will be comparable to the time step for the full simulation. RBM will indeed save computational cost for these cases.

\item Compared with the mean-field regime where RBM is asymptotic-preserving and the variance is controlled uniformly in $N$,  the variance scales like $\mathcal{O}(N^2)$ in the molecular regime.  In fact, the factor $(N-1)/(p-1)$ could make the random variable $\frac{N-1}{p-1}\sum_{j\in\mathscr{C}}K_{2}$ differ a lot in magnitude if $K_2$ changes a lot in magnitude. For example, if one applies RBM to Lennard-Jones fluid with high density where one chooses $r_0$ very small, then $|K_2|$ changes from a large value to a small value from $r_0$ to $L/2$. Then, applying RBM with small batch size like $p=2$ could result in noticeable effects like the numerical heating in molecular dynamics simulations. One has to take some actions like increasing $p$, decreasing $\tau$ or other advanced techniques to reduce such effects (see section \ref{subsec:ljnum}). 
\end{itemize}

Since RBM is asymptotic-preserving and the variance is controlled uniformly in $N$, one may be curious whether we should switch to the mean-field regime and then do the molecular dynamics simulation using RBM. A direct way to obtain the $1/(N-1)$  factor is to do the time scaling $\tilde{t}=(N-1)t$, however this corresponds to zooming in time so it is not quite  the mean-field limit where one should look at the dynamics in a larger scale. In fact, if the interacting forces are homogeneous (like the Coulomb interaction) in space, one should zoom out in both time and space so that the prefactor $1/(N-1)$ can appear.
Though with $1/(N-1)$ factor RBM can work reasonably well for step size not being small, there is no intrinsic change in the physics due to the scaling so applying RBM in the original regime (like molecular regime) does the same thing. 
Scaling can however change the step sizes allowed. For example, if one does the time scaling $\tilde{t}=(N-1)t$, one takes $\mathcal{O}(1)$ step size for the new time variable $\tilde{t}$ while one has to take $\tau=\mathcal{O}(1/N)$ time step in the original molecular regime. This small step size restriction is not due to RBM. In fact, for the full simulation, the step size also has to be small due to the summation of $N-1$ terms. Hence, we will apply RBM directly in the molecular regime when it has benefits, as discussed.

\section{Numerical experiments}\label{sec:numerics}

In this section, we perform some numerical experiments to verify the theoretic claims in section \ref{sec:esterr} and validate RBM with kernel splitting (in particular, Andersen-RBM and Langevin-RBM) via the molecular dynamics simulations for Lennard-Jones fluids. All the numerical experiments in this section are performed via MATLAB R2020a on a Mac Pro laptop with Intel i5-6360U CPU @ 2GHz and 8GB memory.

\subsection{A simple illustrative example}

In this example, we consider an underdamped Langevin equation for $(X^i, V^i)\in \R\times \R$. This example is mainly designed to verify that 
Algorithm \ref{alg:rbm2nd1} works for regular kernels and confirm the theoretical results in section \ref{sec:esterr}.
In particular, we consider the following interacting particle system on $\R$ for $i=1,
\cdots, N$:
\begin{gather}\label{eq:simplelangevin}
\begin{split}
& dX^i=V^i \,dt,\\
& dV^i=-\lambda X^idt+\frac{1}{N-1}\sum_{j: j\neq i}\frac{X^i-X^j}{1+|X^i-X^j|^2}dt
-\gamma V^idt+\sqrt{2\gamma/\beta}\, dW^i.
\end{split}
\end{gather}
The kernel
\[
K(x)=\frac{x}{1+|x|^2}
\]
satisfies $|K|\le \frac{1}{2}$ and $|K'|\le 1$.

\begin{figure}[!htbp]
\centering  
\includegraphics[width=0.5\textwidth]{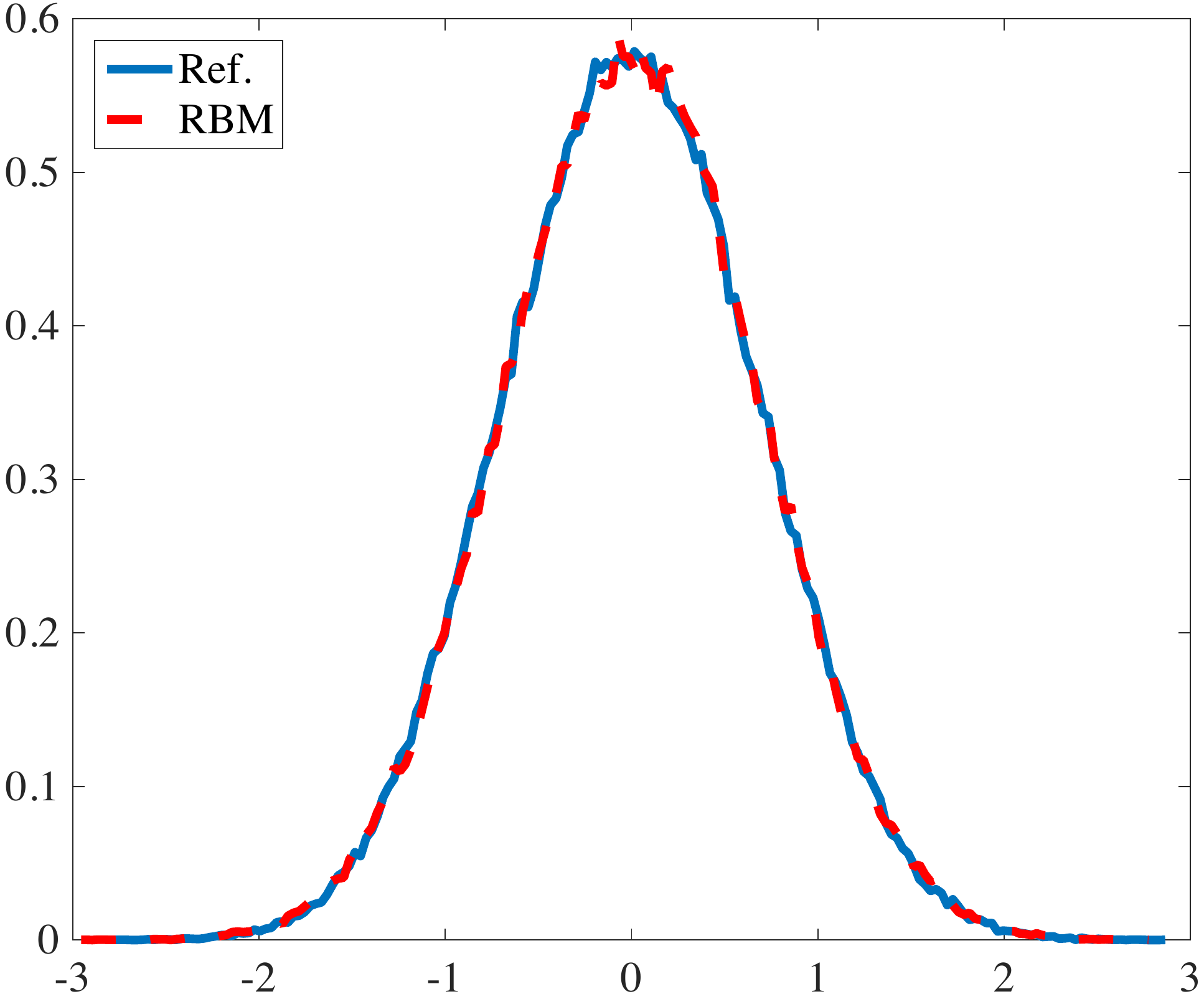}  
\caption{The equilibrium density distribution by RBM (red dashed line) for \eqref{eq:simplelangevin} with $N=500$,  $p=2$, and $\tau=0.02$. The blue solid line is the reference distribution by the full simulation without RBM with step size $\tau=10^{-3}$.} 
 \label{fig:density} 
\end{figure}

Below in the simulations we take
\[
\lambda=\gamma=2.5
\]
so that the conditions in Theorem \ref{thm:longtimeconv} hold, and the temperature is taken as
$\beta^{-1}=1$.
The discretization will be the BAOAB scheme \eqref{baoba}.
The initial positions $X_0^i$'s are sampled i.i.d. from $\mathcal{U}[-0.5, 0.5]$ (the uniform distribution on $[-0.5, 0.5]$), while the initial velocities are also sampled from $\mathcal{U}[-0.5, 0.5]$ but with the empirical mean subtracted $V_0^i\leftarrow V_0^i-\bar{V}_0$ and then the magnitude rescaled such that the average of $(V_0^i)^2$ is the temperature (i.e. $N^{-1}\sum_i (V_0^i)^2=\beta^{-1}$).  For RBM simulations in this example, we always take batch size
\[
p=2.
\]

To verify the effectiveness of RBM,  we first do the simulation for  $N=500$ particles, and check the computed equilibrium distribution.
The system after time $t=50$ is regarded to be in the equilibrium. Hence, we collect the $\{X^i\}$'s from many iterations after $t=50$ as the samples.  In Figure \ref{fig:density}, we show the results by RBM where the Langevin equations are discretized by the ``BAOAB" with step size $\tau=0.02$. We collect the $N=500$ particles as some samples every $\Delta t=0.5$ time (or $25$ iterations for this step size) up to $t=300$. Hence, there are $500*(300-50)/0.5=2.5\times 10^5$ sample points. The reference distribution is plotted using samples at the same time points in the full simulation (i.e. running the Langevin dynamics \eqref{eq:simplelangevin} using ``BAOAB" scheme without RBM) with a step size $\tilde{\tau}=0.001$.  Clearly, the equilibrium distribution density is recovered by RBM with good accuracy.

\begin{table}[!htbp]
\begin{center}
\begin{tabular}{|c|c|c|c|c|}
\hline
\hline
 $$ & $\tau=1$ & $\tau=2^{-1}$ & $\tau=2^{-2}$ & $\tau=2^{-3}$  \\ \hline
 $e^{2x}$ & $0.1098 $ & $   0.0785  $ & $  0.0337  $ & $  9.229\times10^{-5} $  \\ \hline
 $x^2$ & $0.0256  $ & $  0.0252  $ & $  0.0051  $ & $  2.7128\times10^{-4}   $ \\ \hline
$\frac{1}{(x-0.1)^2+0.001}$ & $0.0046  $ & $  0.0016  $ & $  0.0152  $ & $  0.0014  $  \\ \hline
$\frac{1}{1+x^2}$ & $0.0045 $ & $   0.0049  $ & $  6.5742\times10^{-4}  $ & $  0.0016  $  \\
\hline
\hline
\end{tabular}
\medskip
\caption{The weak errors using RBM for equilibrium distribution of \eqref{eq:simplelangevin} with $N=500$.}
\label{t1}
\end{center}
\end{table}

To confirm the sampling correctness quantitatively, we compute the relative weak errors: 
\[
\mathrm{err}_w=\left|\frac{\sum_{i=1}^{N_s} f(X^i)}{N_s}-\frac{\sum_{i=1}^{\bar{N}_s} f(\hat{X}^{i})}{\bar{N}_{s}}\right|\Big/\frac{\sum_{i=1}^{\bar{N}_s} f(\hat{X}^{i})}{\bar{N}_{s}}
\]
 for various test functions $f$. Here, $N_s$ and $\bar{N}_s$ are the numbers of samples for RBM and reference respectively.  In particular, we again run the simulation for $N=500$ with step sizes taken as $\tau=1, 2^{-1},...,2^{-3}$. The samples are again taken every $\Delta t=0.5$ from $t=50$ to $t=300$. 
The samples for the reference are computed using the full simulation with ``BAOAB" scheme and step size $\tilde{\tau}=2^{-10}$.
 The results are listed in Table \ref{t1}, where we take $f(x)=e^{2x}, x^2$, $1/((x-0.1)^2+0.001), 1/(1+x^2)$. Clearly, the weak error in fact tends to zero as we decrease $\tau$ which means RBM indeed can recover the equilibrium distribution. Due to the Monte Carlo fluctuation, the weak convergence order (which should be order $1$ or the weak error scales like $\mathcal{O}(\tau)$ motivated by the results in \cite{jin2020convergence}) is not quite evident in  Table \ref{t1}.

\begin{figure}[!htbp]
\centering  
\includegraphics[width=0.55\textwidth]{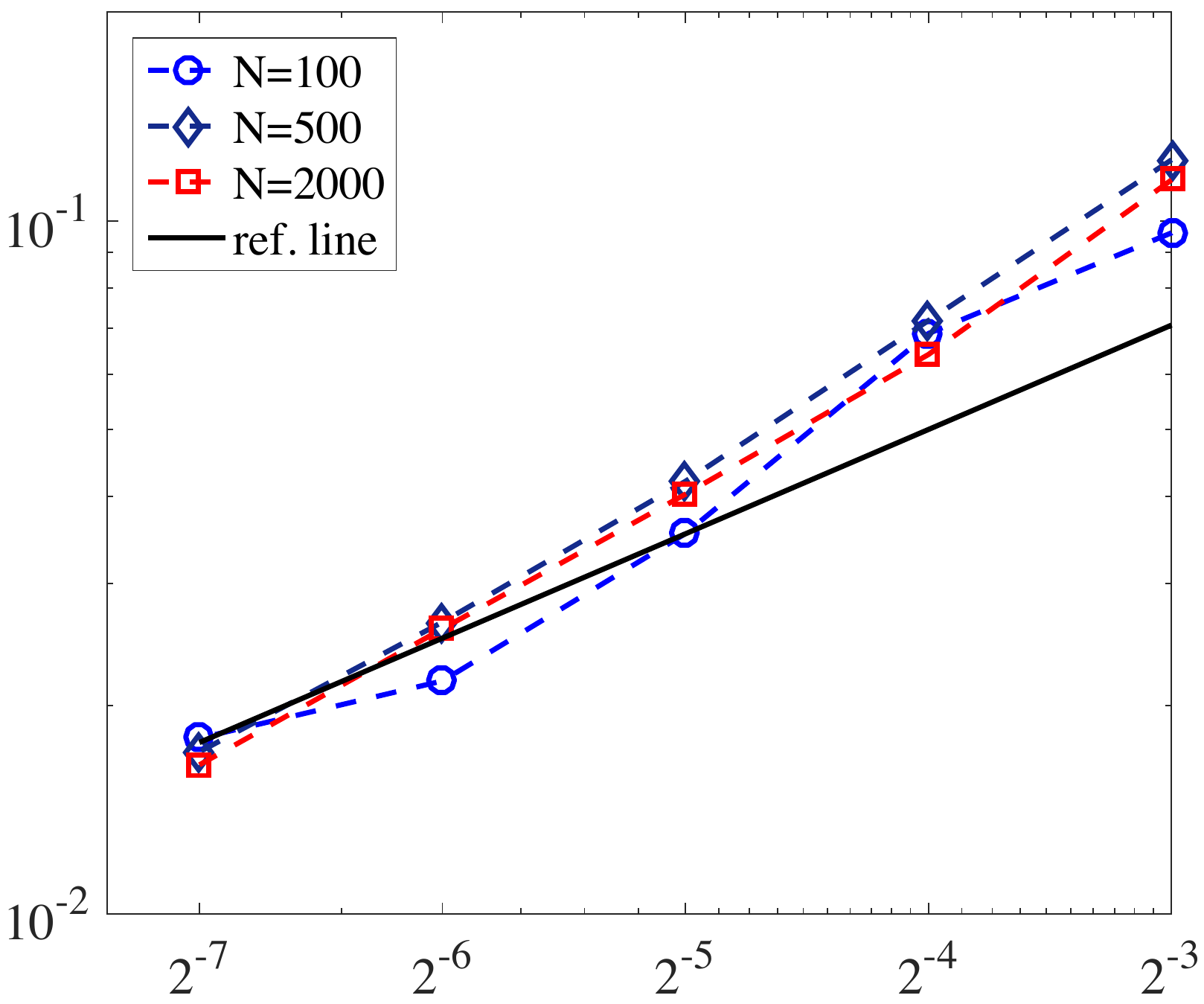}  
\caption{The strong errors $\mathrm{err}_s$ at time $T=2$ v.s. time step size $\tau$ of RBM for \eqref{eq:simplelangevin} with different system sizes. The black solid line is $E=0.2\tau^{1/2}$ for reference. } 
 \label{fig:errs} 
\end{figure}

To verify the strong convergence order claimed in Theorem \ref{thm:longtimeconv},  we consider the relative strong errors: 
\[
\mathrm{err}_s=\sqrt{\frac{\sum_{i=1}^N (X^i-\hat{X}^{i})^2}{N}\Big/\frac{\sum_{i=1}^N ({\hat{X}^{i}})^2}{N}},
\]
 where $X^i$'s are numerical solutions by RBM at $T=2$ and $\hat{X}^{i}$'s are the reference solutions. 
 
 The results are shown in Figure \ref{fig:errs} for sizes $N=50, 500, 2000$. 
 The reference solution is obtained using the full batch (the original particle system) using ``BAOAB" with step size $\tau=2^{-18}$. The same Brownian motions are used for the reference solution and RBM solution (i.e., the realizations of  Brownian Motions  used for the reference solution are stored and then applied for the RBM simulation) for the strong solution. 
The simulation results above indicate that RBM can indeed obtain the $1/2$ strong order for the underdamped Langevin equations with regular kernels, agreeing with our theory in Theorem \ref{thm:longtimeconv} (in this example, the strong order seems slightly better than $1/2$ for the steps we consider). Moreover, the weak error above indicate that the equilibrium can be correctly captured by RBM as well, consistent with the claim in the theorem that the error control is uniform in time.

\subsection{The Lennard-Jones fluid}\label{subsec:ljnum}

In this section, we test RBM with splitting, especially the methods discussed in section \ref{sec:md}, on the Lennard-Jones fluids to validate these methods. In fact, our experience indicates that applying Algorithm \ref{alg:rbm2nd1} directly to such systems truly yields numerical instability, so RBM with splitting is desired for such applications. The potential $\phi$ in \eqref{eq:md1} is then given by ($x\in \R^3$) in this setting
\begin{gather}\label{eq:ljpotential}
\phi(x)=4\left(\frac{1}{r^{12}}-\frac{1}{r^6}\right),\quad r=|x|.
\end{gather}

As mentioned in section \ref{sec:md}, the periodic boxes are used to approximate the fluids of large extent. Let $L$
be the length of the box. With periodic setting, a particle interacts with not only another particle but also its periodic images. Thanks to the fast decaying properties of the Lennard-Jones potential, one can pick a cutoff length $r_c$ so that the interaction between two particles (including particle-image interaction) with distance larger than $r_c$ will be treated in a mean-field fashion (see \cite[Chap. 3]{frenkel2001understanding} for more details). Following \cite[Chap. 3]{frenkel2001understanding}, we choose
\[
r_c=L/2.
\]
With the cutoff mentioned here, the pressure formula is given approximately by (\cite[section 3.4]{frenkel2001understanding},\cite[section 4.2]{lixuzhao2020}):
\begin{align}\label{eq:pressureforperiodic}
P=\rho T+\frac{8}{V}\sum_{i=1}^N \sum\limits_{j: j>i,\tilde{r}_{ij}<r_c}(2\tilde{r}_{ij}^{-12}-\tilde{r}_{ij}^{-6})+\dfrac{16}{3}\pi\rho^2 \left[\dfrac{2}{3}\left(\dfrac{1}{r_c}\right)^9-\left(\dfrac{1}{r_c}\right)^3\right],
\end{align}
where $T$ is the scaled temperature, $V=L^3$ is the volume, and we have introduced
\[
\tilde{r}_{ij}=|\vec{r}_{ij}+\vec{n}L|
\]
for some suitable three-dimensional integer vector $\vec{n}$ so that
$|\vec{r}_{ij}+\vec{n}L|$ is minimized. Note that since $r_c=L/2$, then for each $j$, there is then at most one image (including itself) that falls into $B(x_i, r_c)$.  Hence, when implementing the methods, the forces between particles are computed using the nearest image (see \cite[sec. 4.2.2]{frenkel2001understanding}).

To apply the methods in section \ref{sec:md}, in all the simulations below, we take $r_0=\sqrt[6]{2}$ and split the potential $\phi$ 
\begin{gather}\label{eq:LJsplit1}
\phi(x)=:\phi_1(x)+\phi_2(x),
\end{gather}
where
\begin{equation}\label{eq:LJsplit3}
\phi_1(x)=\begin{cases}
4(\frac{1}{r^{12}}-\frac{1}{r^{6}})+1, & 0<r<\sqrt[6]{2},\\
0, & r\geq\sqrt[6]{2},
\end{cases}
\end{equation}
and
\begin{equation}\label{eq:LJsplit2}
\phi_2(x)=\begin{cases}
-1, & 0<r<\sqrt[6]{2},\\
4(\frac{1}{r^{12}}-\frac{1}{r^{6}}), & r\geq\sqrt[6]{2}.
\end{cases}
\end{equation}
The force $K=-\nabla\phi$ is split correspondingly.  That means the part of interaction force for $r\le \sqrt[6]{2}$ is regarded to have short range and the part for $r>r_0=\sqrt[6]{2}$ is regarded to have long range. The long range parts ($-\nabla \phi_2$) will be computed using RBM. Note that the threshold $r_0=\sqrt[6]{2}$ is different from the cutoff $r_c=L/2$ above. The cutoff $r_c$ above means that the molecular interactions are computed explicitly only for $r\le r_c$ while the ones for $r\ge r_c$ are treated in a mean-field fashion. 

\begin{remark}\label{rmk:ljcut}
Using the cutoff $r_c=L/2$ in simulation yields $\mathcal{O}(N^2)$ complexity for direct simulation. Since
the potential decays very fast and the density considered in this section is $\mathcal{O}(1)$, one may use smaller cutoff $r_c$
and obtain roughly correct results, and this strategy will reduce the complexity to roughly linear as well. However, our experience indicates that applying RBM with $r_0=\sqrt[6]{2}$ still saves as the batch size $p$ (which is $2$ in the experiments below) is much less than the effective number of particles that interact with one particular particle. (In this paper, we mainly aim to validate the effectiveness of RBM so we did not compare the computational time of RBM with those for some current methods for Lennard-Jones fluids.)
\end{remark}

\begin{figure}[!htbp]
\centering
\subfigure[$N=100$.]{
\begin{minipage}[b]{0.47\textwidth}
\includegraphics[width=1\textwidth]{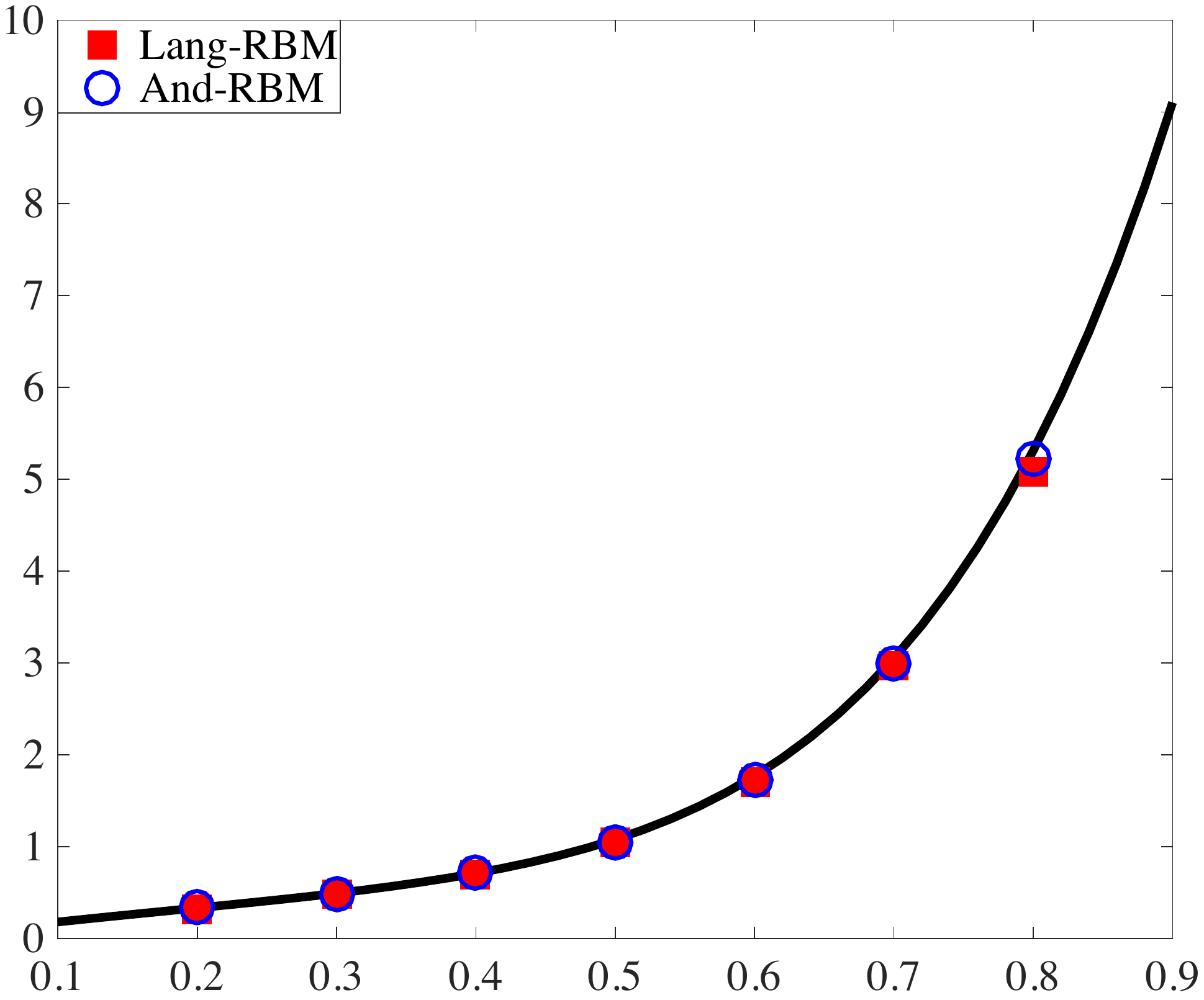}
\end{minipage}
}
\subfigure[$N=500$.]{
\begin{minipage}[b]{0.47\textwidth}
\includegraphics[width=1\textwidth]{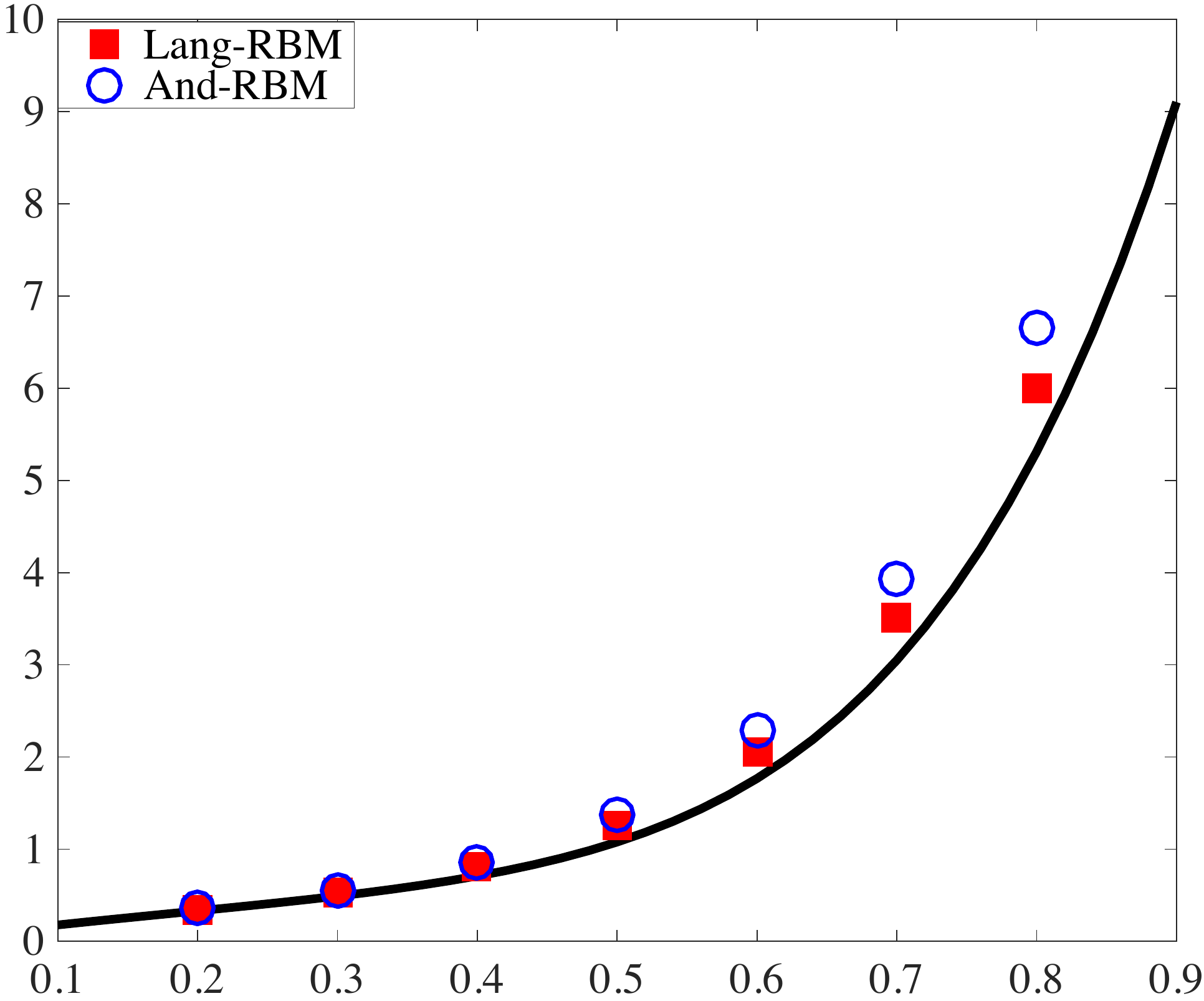}
\end{minipage}
}
\caption{The pressure computed by Andersen-RBM and Langevin-RBM for Lenard-Jones fluid. The black solid line is the  reference fitting curve in \cite{johnson1993lennard}. The blue circles are the pressure computed by the Andersen-RBM and  the red squares are by Langevin-RBM. $\nu=10$; $\gamma=10$; the time is $T=50$; the step size is $\tau=0.001$.}
\label{fig:lgvadr001}
\end{figure}

For the simulation, the temperature is taken to be $\beta^{-1}=2$ and the length of the box is set as $L=(N/\rho)^{1/3}$ for a given density $\rho$. The particles are initially put on the cubic lattice with grid size $L/N^{1/3}$. The initial velocities are randomly chosen from uniform distribution $\mathcal{U}^3[-0.5,0.5]$, and then shifted and rescaled so that the instantaneous temperature matches desired value (i.e., $N^{-1}\sum_i |V_0^i|^2=3\beta^{-1}$). 

For the thermostats, we use both the Andersen thermostat and the underdamped Langevin dynamics for simplicity as indicated in section \ref{subsec:heatbath}, and the resulted schemes are the "Andersen-RBM" and "Langevin-RBM" as already explained in section \ref{subsec:heatbath}. The batch size is taken as $p=2$ for all the experiments here.

We first run the simulations with the collision coefficient $\nu=10$ for Andersen-RBM and the friction coefficient $\gamma=10$ for Langevin-RBM. The simulation before time $T=50$ is regarded as the burn-in phase, and we compute the pressure using the viral formulation \eqref{eq:pressureforperiodic} at a given time point (after $T=50$). We compute $10^5$ such pressures (each for one iteration) and then take the average as the computed the value. 
The computed values using $N=100$ and $N=500$ are shown in Fig. \ref{fig:lgvadr001} for various densities.
The reference curve (black solid line) is the fitting curves in \cite{johnson1993lennard}. 
The results show that RBM with splitting strategy \eqref{eq:LJsplit1}-\eqref{eq:LJsplit3} can work reasonably well for the Lennard-Jones fluid in the considered regime. Our experience indicates that direct application of RBM without splitting (Alg. \ref{alg:rbm2nd1})  to the Lennard-Jones potential \eqref{eq:ljpotential} indeed results in numerical instability.

\begin{figure}[!htbp]
\centering
\subfigure[$\nu=\gamma=10$, $\tau_k=0.001/\log(k+1)$.]{
\begin{minipage}[b]{0.47\textwidth}
\includegraphics[width=1\textwidth]{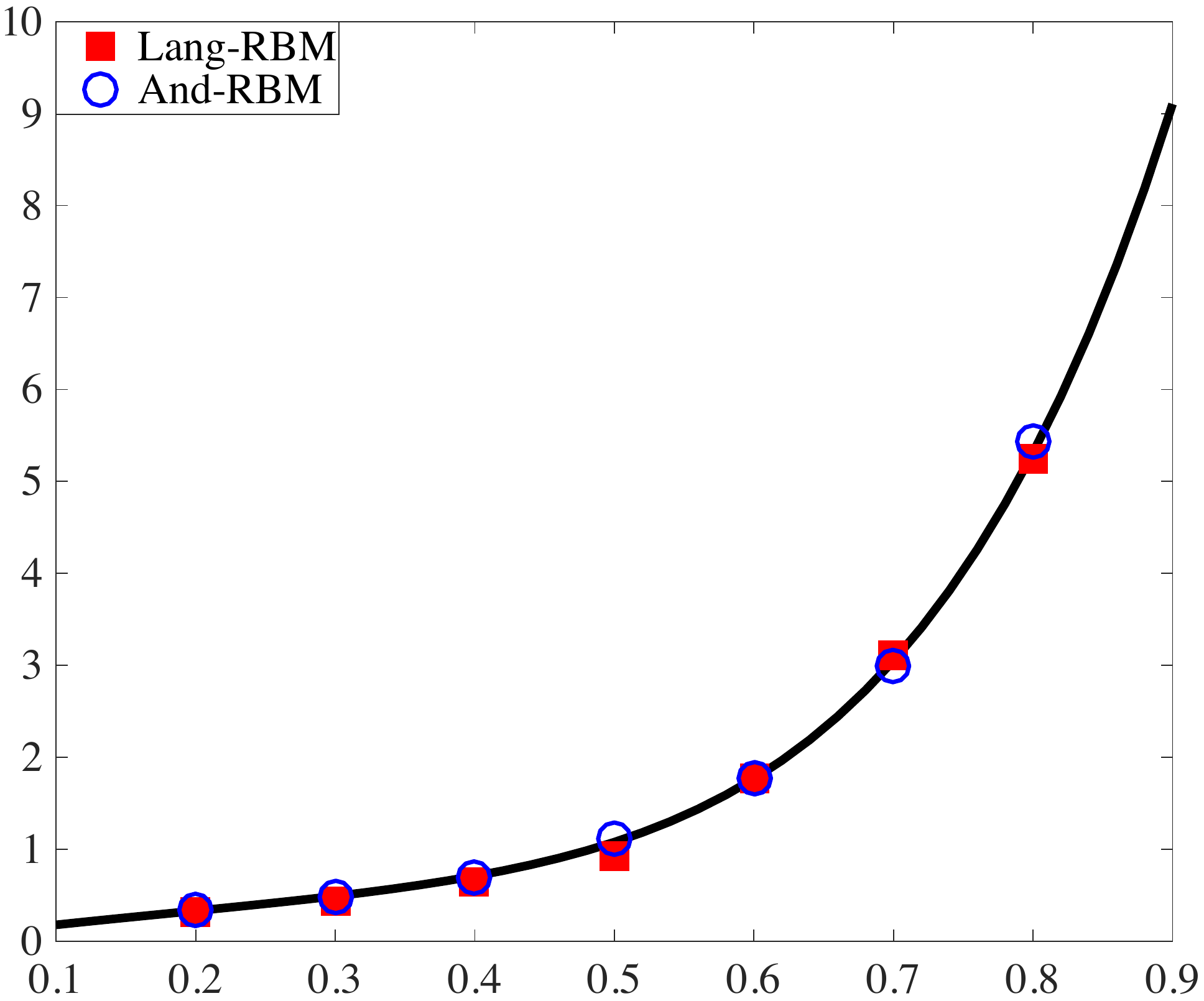}
\end{minipage}}
\subfigure[$\nu=\gamma=50$, $\tau=0.001$.]{
\begin{minipage}[b]{0.47\textwidth}
\includegraphics[width=1\textwidth]{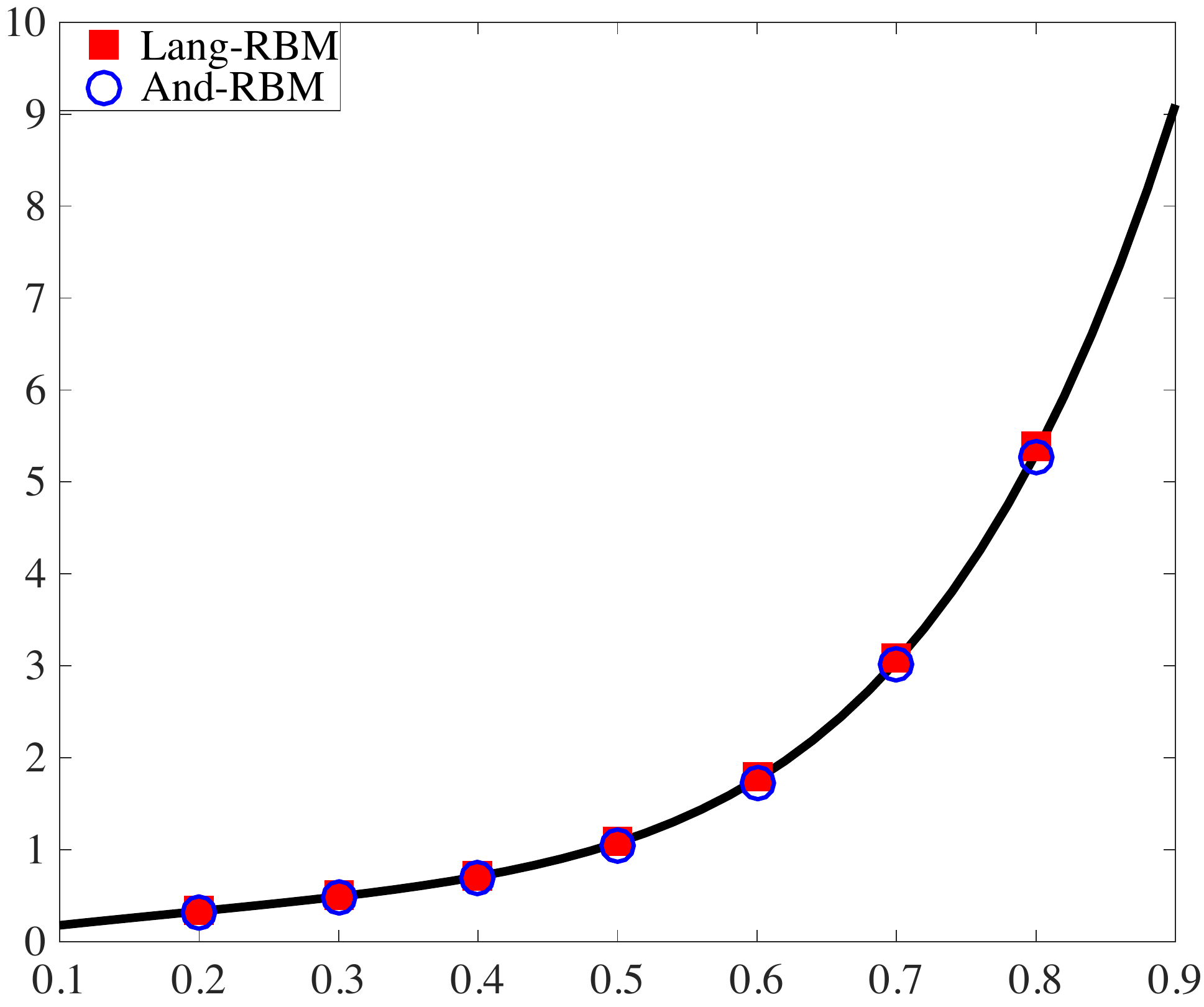}
\end{minipage}
}
\caption{The pressure obtained by Andersen-RBM and Langevin-RBM using two strategies to reduce numerical heating for Lenard-Jones fluid with $N=500$: the blue circles are those by Andersen-RBM while the red squares are by Langevin-RBM.}
\label{fig:lgvadr}
\end{figure}

Another observation from Fig.~\ref{fig:lgvadr001} is that when $N=500$ the extra variance brought by RBM can result in noticeable numerical heating and thus bigger pressure (the variance depends on $N$ because we are working in the molecular regime, and see the discussion in section \ref{subsec:dis}).  To reduce the numerical heating or increase the temperature control ability, we try two strategies. 
The first strategy is to decrease the step size $\tau_k$ as the iteration goes on to decrease the numerical heating. This idea is similar to the one in simulated annealing \cite{MR904050,hwang1990large}. The second strategy is to increase the friction coefficient (i.e., the $\nu$ and $\gamma$ in Andersen thermostat and Langevin dynamics respectively). Increasing the collision or friction coefficient clearly makes the system relax faster to the quasi-equilibrium, but it may also bring in some unphysical effects \cite{hoover1985canonical,frenkel2001understanding}.
We show the numerical results in Figure \ref{fig:lgvadr}. Figure \ref{fig:lgvadr} (a) shows the results using the first strategy (i.e., decreasing step size), where we take
\[
\tau_k=0.001/\log(k+1).
\]
Figure \ref{fig:lgvadr} (b) shows the results using the second strategy (i.e., using larger friction coefficient so that the temperature control is better) where we take $\gamma=\nu=50$. 
Clearly, after these two approaches are applied, the numerical heating effects are reduced significantly, and the correct equation of states is obtained.  As another possible thermostat for better temperature control, one may consider Nos\'e-Hoover thermostat \cite[Chap. 6]{frenkel2001understanding}.

\begin{figure}[!htbp]
\centering  
\includegraphics[width=0.55\textwidth]{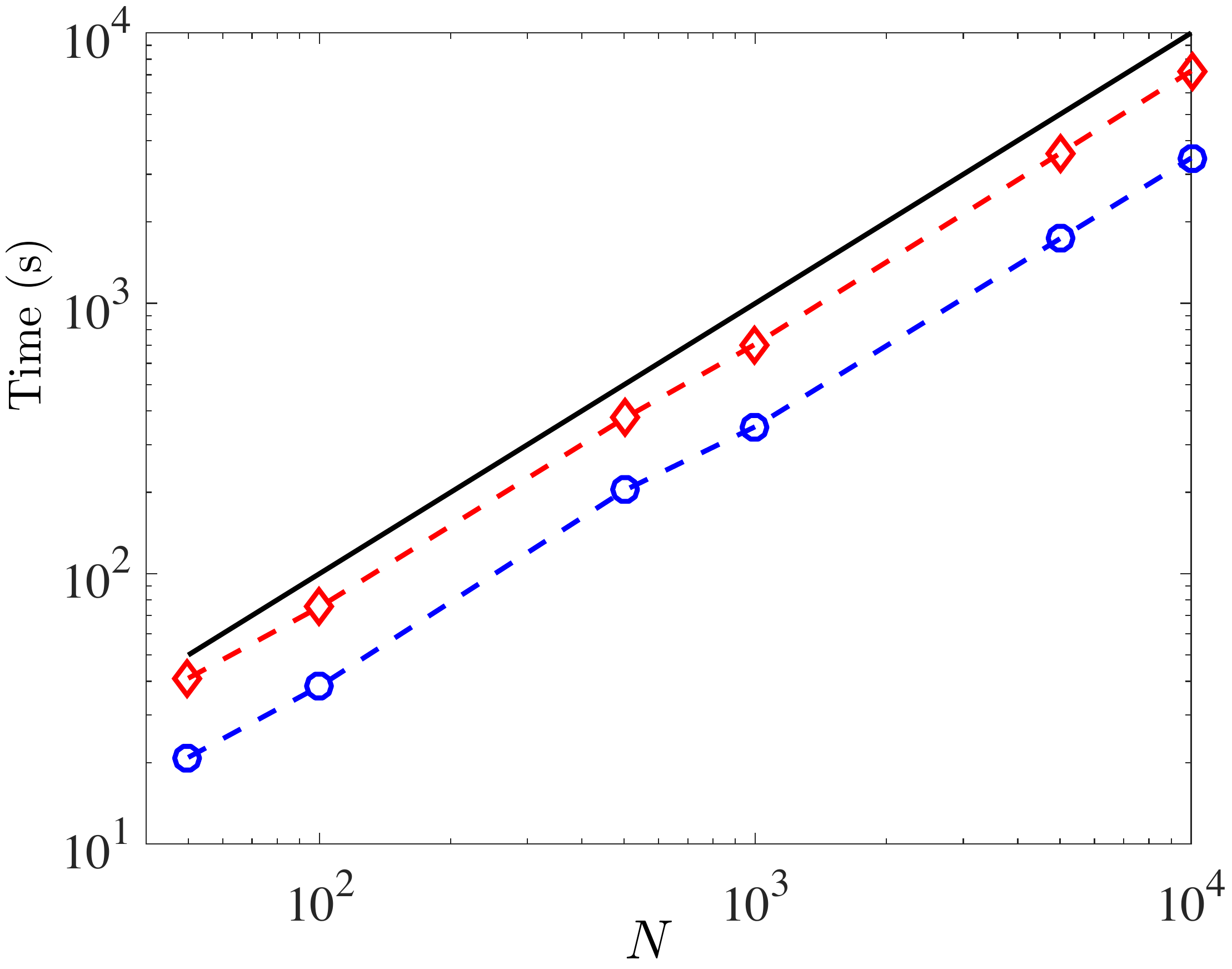}  
\caption{The CPU time v.s. size of the system for RBM-Andersen (blue circles) and RBM-Langevin (red diamonds). Clearly, the computational time scales linearly with the size of the system for both methods} 
 \label{fig:cpu} 
\end{figure}

Lastly, we validate the claim that the complexity of our algorithm is $\mathcal{O}(N)$ in Figure \ref{fig:cpu}, where 
the CPU time is plotted versus the size of the Lennard-Jones system. The simulation is performed up to time $30$ with step size $\tau=2^{-10}$ for systems with density $\rho=0.5$. Clearly, both Andersen-RBM and Langevin-RBM scale linearly with the size of the system, and this result thus verified our claim.

\section*{Acknowledgement}
S. Jin was partially supported by the NSFC grant No. 31571071.  The work of L. Li was partially sponsored by NSFC 11901389, 11971314, and Shanghai Sailing Program 19YF1421300. L.L. would like to thank Prof. Xiantao Li  for suggestions on the Lennard-Jones fluid simulations.

\bibliographystyle{plain}
\bibliography{sdealg}

\end{document}